\documentclass[11pt]{amsart}

\newcommand{\Ipres}[2]{\mathrm{Inv}\langle #1\:|\:#2 \rangle}
\newcommand{\Gpres}[2]{\mathrm{Gp}\langle #1\:|\:#2 \rangle}

\usepackage[top=1.5in, bottom=1.5in, left=1in, right=1in]{geometry}
\usepackage{amsgen}
\usepackage{amsmath}
\usepackage{amstext}
\usepackage{amsbsy}
\usepackage{amsopn}
\usepackage{amsfonts}
\usepackage{amssymb}
\usepackage{graphicx}
\usepackage{xcolor}
\usepackage{wrapfig}

\def\mapright#1{\smash{\mathop{\longrightarrow}\limits^{#1}}}
\def\A{{\mathcal{A}}}
\def\Rw{\Rightarrow}

\def\invp{{\rm Inv}}
\def\P{{\mathcal{P}}}

\def\S{{\mathcal{S}}}
\def\p{\varphi}
\def\inv{^{-1}}

\newcommand{\FIM}{{\rm FIM}}
\newcommand{\FG}{{\rm FG}}
\newcommand{\MT}{{\rm MT}}
\renewcommand{\exp}{{\rm Exp}} 
\newcommand{\Exp}{{\rm Exp}}
\newcommand{\diam}{{\rm diam}} 
\def\x0#1{\Gamma(x_0, {#1})} 
\def\cc0#1{\Gamma^c(x_0, {#1})} 
\def\ck#1#2{\Gamma(x_0, {#2})} 
\def\cck#1#2{\Gamma^c(x_0, {#2})} 
\def\tc0#1{\tilde{\Gamma}(x_0, {#1})} 
\def\tcc0#1{\tilde{\Gamma}^c(x_0, {#1})} 

\newcommand{\beq}{\begin{equation}}
\newcommand{\eeq}{\end{equation}}
\def\bi{\begin{itemize}}
\def\ei{\end{itemize}}

\newtheorem{T}{Theorem}[section]
\newcommand{\bt}{\begin{T}}
\newcommand{\et}{\end{T}}
\newcommand{\ftd}{$\square$\end{T}}

\newtheorem{Proposition}[T]{Proposition}
\newcommand{\bp}{\begin{Proposition}}
\newcommand{\ep}{\end{Proposition}}
\newcommand{\fpd}{$\square$\end{Proposition}}

\newtheorem{Lemma}[T]{Lemma}
\newcommand{\bl}{\begin{Lemma}}
\newcommand{\el}{\end{Lemma}}
\newcommand{\fld}{$\square$\end{Lemma}}

\newtheorem{Corol}[T]{Corollary}
\newcommand{\bc}{\begin{Corol}}
\newcommand{\ec}{\end{Corol}}
\newcommand{\fcd}{$\square$\end{Corol}}

\newtheorem{Result}[T]{Result}
\newcommand{\br}{\begin{Result}}
\newcommand{\er}{\end{Result}}
\newcommand{\frd}{$\square$\end{Result}}

\newtheorem{Example}[T]{Example}
\newcommand{\be}{\begin{Example}}
\newcommand{\ee}{\end{Example}}

\newtheorem{Problem}[T]{Question}
\newcommand{\bq}{\begin{Problem}}
\newcommand{\eq}{\end{Problem}}

\newtheorem{Remark}[T]{Remark}
\newcommand{\brm}{\begin{Remark}}
\newcommand{\erm}{\end{Remark}}

\newcommand{\bid}{\begin{Remark}}
\newcommand{\eid}{\end{Remark}}

\newtheorem{Def}[T]{Definition}

\renewcommand{\proof}
   {\par\medbreak\noindent{\bf Proof}.\enspace}

\renewcommand{\qed}{
$\Box$
\par\bigbreak}

\begin{document}

\title[Inverse monoids with hyperbolic and tree-like Sch\"utzenberger graphs]{Algorithmic properties of inverse monoids\\ with hyperbolic and tree-like Sch\"utzenberger graphs}

\author{Robert D. Gray}
\address{School of Mathematics, University of East Anglia\\ Norwich NR4 7TJ, England, UK} 
\email{robert.d.gray@uea.ac.uk}
\author{Pedro V. Silva}
\address{Centro de Matem\'atica, Universidade do Porto\\
	R. Campo Alegre 687, 4169-007 Porto, Portugal}
\email{pvsilva@fc.up.pt}
\author{N\'{o}ra Szak\'{a}cs}
\address{Department of Mathematics, University of York\\ YO10 5DD
	United Kingdom}
\address{Bolyai Institute,
	University of Szeged\\
	Aradi v\'ertan\'uk tere 1.
	H-6720 Szeged, Hungary}
\email{nora.szakacs@manchester.ac.uk}

\subjclass[2010]{20F10, 20F05, 20M05, 20M18, 20F36}

\keywords{word problem, 
	finitely presented inverse monoid, 
	hyperbolic group, 
	virtually free group, context-free language, tree-like inverse monoid}

\thanks{The second author was partially supported by CMUP, which is financed by national funds through FCT -- Funda\c c\~ao para a Ci\^encia e a Tecnologia, I.P., under the project with reference UIDB/00144/2020. }
\thanks{
\noindent
	\setlength\intextsep{0pt}
	\begin{wrapfigure}{l}{0cm}
		\includegraphics[height=3em]{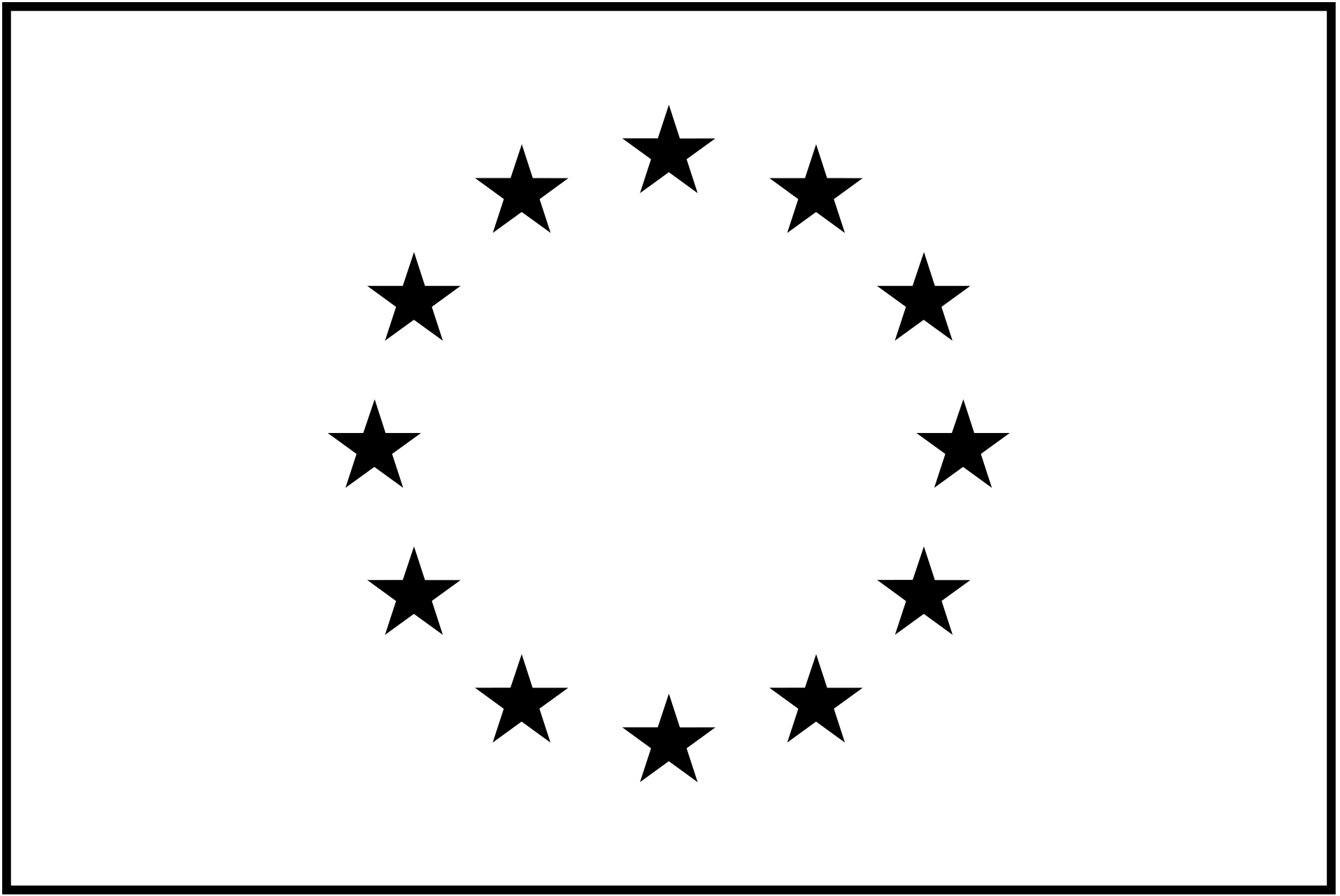}
	\end{wrapfigure}
	The third author was funded by the European Union’s
	Horizon 2020 research and innovation programme under the
	Marie Sk{\l}odowska-Curie grant agreement No 799419, and was partially supported by the Hungarian National Foundation for Scientific Research grant nos. K115518 and K128042.
}




\begin{abstract} 
 We prove that the class of finitely presented inverse monoids whose Sch\"utzenberger graphs are quasi-isometric to a trees has a uniformly solvable word problem, furthermore, the languages of their Sch\"utzenberger automata are context-free. On the other hand, we show that there is a finitely presented inverse monoid with hyperbolic Sch\"utzenberger graphs and an unsolvable word problem.
\end{abstract}

\maketitle

\section{Introduction}

In the 1910s, Dehn proved that fundamental groups of closed
orientable two-dimensional manifolds have a solvable word problem. The crucial feature of
these groups which the proof depended on led to the notion of small cancellation properties. In
the late 80s, Gromov introduced hyperbolic groups, motivated by giving small cancellation
a geometric interpretation. The notion represented a revolution in group theory due to a
conjunction of factors:
\begin{itemize}
	\item hyperbolic groups can be characterized using a geometric property of their Cayley
	graphs, called the Rips condition;
	\item hyperbolic groups have excellent algorithmic properties: they are biautomatic (in particular,
	they have an effectively solvable word problem), and their geodesics constitute
	an automatic structure;
	\item hyperbolic groups are finitely presented.
\end{itemize}
Hyperbolic groups led to the emergence of the vibrant field of geometric group theory and continue to be one of its central subjects, see for instance the recent monograph \cite[Chapter 11]{KapDru} for details.

Such a beautiful theory arouses inevitably the desire of extending it to other algebraic
structures, in particular, to (some class of) semigroups. Presentations of semigroups and the
corresponding word problem can be analogously defined.
So can the Cayley graph of a semigroup $S=\langle A \rangle$: it is a directed, edge-labeled graph with vertex set $S$, and for any $s \in S$ and $a \in A$, an edge labeled by $a$ from $s$ to $sa$.
The obstacle
to extending hyperbolicity in a meaningful way is that the geometric properties of the Cayley
graph of a semigroup are not as closely tied to its algebraic properties as in the case of groups.
Duncan and Gilman \cite{DG} have proposed a notion of hyperbolic semigroup which generalizes a
language-theoretic condition characterizing hyperbolic groups. But this condition is not at
all geometric, and that constitutes a handicap for a notion that had a geometric motivation
from the beginning. On the other hand, Cassaigne and Silva [3] considered Rips condition for
the undirected Cayley graph of monoids defined by special confluent rewriting systems and
used it to get a hyperbolic boundary. But this only settles a small subclass of monoids.

In this paper we introduce a new approach to the study of hyperbolicity and related geometric conditions in \emph{inverse monoids}. Inverse monoids are one of the most important generalizations of groups, that originally emerged as the abstract counterparts of partial symmetries. In an inverse monoid, every element $s$ has a unique inverse with the property $ss^{-1}s=s$ and $s^{-1}ss^{-1}=s^{-1}$.
A rich combinatorial and geometric theory for inverse
monoid presentations emerged in the early 90s due to the work of Stephen, Margolis
and Meakin \cite{Ste1, MM, MM1}, among others. The Cayley graph of an inverse monoid $S$ generated by $A$ (as an inverse monoid) is just the Cayley graph of the semigroup $S$ with respect to the semigroup generating set $A \cup A^{-1}$. The Cayley graph of an inverse monoid is in general not strongly connected (in particular words of the form $aa^{-1}$ may not label loops), and the geometry of the undirected Cayley graph does not adequately capture the properties of the monoid -- for instance adjoining a $0$ to any inverse monoid leaves its algorithmic properties unchanged, but makes its undirected Cayley graph of diameter at most $2$.

The strongly connected components of the (directed) Cayley graph are called Sch\"utzenberger graphs, and these have proven to be very useful in investigating algorithmic questions. One the one hand, Sch\"utzenberger graphs share key
properties with Cayley graphs of groups, in particular, the word metric defines a metric on each Sch\"utzenberger graph.
On the other hand, the whole family of Sch\"utzenberger
graphs determines the inverse
monoid uniquely, and the respective notion of Sch\"utzenberger automaton is the key to the solution to the word problem \cite{Ste1}. 

There are numerous results in the literature which suggest that if the Sch\"utzenberger graphs of a finitely presented inverse monoid are sufficiently tree-like, then the inverse monoid will have good algorithmic properties; see e.g. 
\cite{BMM, CMP, HLM, LO, MM, RC}.
The broadest class that one might hope to establish this for in general would be the inverse monoids whose Sch\"utzenberger graphs are Gromov-hyperbolic.  However in Section \ref{Sec.hyperbolic} we shall give a construction which proves, rather surprisingly, that there exist such finitely presented inverse monoids with an undecidable word problem. In fact such examples exist where all the Sch\"utzenberger graphs are $\delta$-hyperbolic for a single fixed value of $\delta$. This shows that in order to obtain a class of inverse monoids with good algorithmic properties a condition stronger than hyperbolicity is needed.

In Section \ref{Sec.treelike}, we impose the stronger geometric condition on Sch\"utzenberger graphs of polygon hyperbolicity, which is equivalent to being quasi-isometric to trees. We call such finitely generated inverse monoids tree-like. It follows from Stallings' theorem about ends of groups and Dunwoody's accessibility theorem that tree-like groups are exactly the finitely generated virtually free groups \cite[Theorem 20.45]{KapDru}, while the celebrated Muller-Schupp theorem \cite{MS} states that they are exactly the groups with a context-free word problem. In the language of inverse monoids, this last condition on groups is equivalent to the language of each Sch\"utzenberger automaton being context-free.

We find that in the case of inverse monoids, even being tree-like does not imply finite presentability (Corollary \ref{nonfp}). However finitely presented tree-like inverse monoids share some of the nice algorthimic properties of hyperbolic and virtually free groups, in particular 
\begin{itemize}
	\item each Sch\"utzenberger graph has a rational set of geodesics (Theorem \ref{phrat}),
	\item each Sch\"utzenberger automaton has a context-free language (Theorem \ref{cfthm}), 
	\item the word problem is uniformly solvable (Theorem \ref{thm:solvwp}).
\end{itemize}
We note that we can derive from Theorem \ref{cfthm} analogous results proved in the literature for special subclasses of tree-like inverse monoids \cite{AR, CNR}.

We close the paper with a section on open questions motivated by the rich theories of virtually free and hyperbolic groups and our results.

\section{Preliminaries}

\subsection{Graphs}
Most of the graphs considered in the paper are edge-labeled, directed graphs, and these will just be referred to as graphs for brevity. Almost all of these graphs have edges occurring in inverse pairs, i.e. they are graphs in the sense of Serre \cite{Serre}. This property of graphs is called \emph{symmetric} in the paper. Occasionally, undirected graphs also come up, in these cases we will emphasize that they are undirected. By a tree, depending on context, we either mean an undirected tree in the usual sense, or a digraph obtained from an undirected tree by replacing the edges with a pair of opposite directed edges.

For a subset $S$ of any graph $\Gamma$, the subgraph induced by $S$ is denoted by $\langle S \rangle$. If $\Gamma$ is a symmetric graph, the induced subgraph
is also defined to be symmetric, containing edges in inverse pairs.

Given a graph $\Gamma$, the set of its vertices is denoted by $V(\Gamma)$, the set of its edges by $E(\Gamma)$. For any edge $e \in E(\Gamma)$, $\alpha(e)$ denotes its initial vertex, $\omega(e)$ its terminal vertex. The graph $\Gamma$ is considered to be the union of its edges and vertices, that is, the notation $t \in \Gamma$ means $t$ is either an edge or a vertex of $\Gamma$. The label of an edge $e$ is denoted by $l(e)$.

A \emph{path} in a graph is a sequence of consecutive edges, i.e. a sequence $e_1 \ldots e_n$ of edges where $\omega(e_i)=\alpha(e_{i+1})$ for $1 \leq i \leq n-1$. Given a path $e_1 \ldots e_n$, we call $n$ the length of the path $(n\geq 0)$. The initial and terminal vertices of a path $p$ will also be denoted by $\alpha(p)$, and respectively $\omega(p)$. A label of the path $p=e_1 \ldots e_n$ is the word
$l(e_1) \ldots l(e_n)$, and it is denoted by $l(p)$. The path $p$ is called \textit{closed} if $\alpha(p)=\omega(p)$. A path is called \emph{simple} if all the vertices it traverses are different, with the possible exception of the initial and terminal vertex. {As usual, the concatenation of two paths $p$ and $q$ with $\alpha(q)=\omega(p)$ is denoted by $pq$.}

One can define a distance function on the set of vertices of a connected, symmetric graph, or on the vertices of an undirected graph by 
$$d(u,v)=\min \{n: e_1 \ldots e_n\hbox{ is a path from }u\hbox{ to }v\}.$$
Note that in the case of general (directed) graphs, this function may not be symmetric or finite. In the following, we consider the vertex sets of symmetric graphs as metric spaces defined by the metric above. 

Let $\Gamma$ be any symmetric digraph. A \emph{geodesic path} in $\Gamma$ is a path $p$ such that the length of $p$ is exactly $d(\alpha(p),\omega(p))$. In particular, any non-loop edge is a geodesic.
For vertices $u$ and $v$ of $\Gamma$, we sometimes use $[u,v]$ to denote a geodesic path from $u$ to $v$.

For any subgraph $X$ of $\Gamma$, denote by $X^{+r}$ the graph induced by those vertices of $\Gamma$ which are at distance at most $r$ from $X$. We refer to this subgraph as the $r$-neighborhood of $X.$ In particular, for a vertex $x_0$, $x_0^{+r}$ is the closed disk of center $x_0$ and radius $r$, which we denote by $D_r(x_0)$. The diameter of a connected subgraph $X$ is 
$$\diam(X)=\sup\{d(u,v): u, v \in V(X)\}.$$

\subsection{Automata and languages}

In this section we provide a very brief introduction to automata and formal languages. A more comprehensive exposition can be found for instance in \cite{HU}.

An {\em automaton} is a structure of the form $\A = (\Gamma, I,T)$ where $\Gamma$ is an edge-labeled digraph, and $I,T \subseteq V(\Gamma)$ are nonempty subsets called initial and terminal vertices respectively.  We also assume that there is at most one edge in $\Gamma$ with any given initial and terminal vertex and label.
Vertices of $\Gamma$ are often referred to as states, and edges as transitions. The set of labels is called the alphabet, let us denote it by $A$.
The automaton $\A = (\Gamma, I,T)$ is finite if both $V(\Gamma)$ and $E(\Gamma)$ are finite. We say that $\A$ is {\em deterministic} if $|I| = 1$ and for any $e_1, e_2 \in E(\Gamma)$, if we have $\alpha(e_1)=\alpha(e_2)$ and $l(e_1)=l(e_2)$, then $\omega(e_1)=\omega(e_2)$. {When $\Gamma$ is a symmetric, connected graph labeled over an alphabet of the form $A=B \cup B^{-1}$ such that $l(e^{-1})=(l(e))^{-1}$ for any edge $e$, and $\A$ is deterministic, we call $\A$ an \emph{inverse automaton}.}

The set of all finite words over the alphabet $A$ are denoted by $A^\ast$, this forms the free monoid on $A$ with respect to concatenation. The identity element is the empty word, which is denoted by $\epsilon$. Subsets of $A^\ast$ are called  \textit{languages} over $A$.

A path $p$ in $\A$ is {\em successful} if $\alpha(p)$ is initial and $\omega(p)$ is terminal. This can be interpreted as the automaton accepting the input word $l(p)$. The {\em language $L(\A)$ of $\A$} is the set of labels of successful paths in $\A$, that is, the set of words $\A$ accepts. 
A language $L \subseteq A^*$ is called {\em rational} if it is the language of some finite automaton.

Two automata $\A = (\Gamma,I,T)$ and $\A' = (\Gamma',I',T')$ are {\em isomorphic} if there exists a labeled graph isomorphism $\varphi:\Gamma \to \Gamma'$ such that $\varphi(I) = I'$, $\varphi(T) = T'$.
We use the notation $\A \cong \A'$ to express that the automata $\A$ and $\A'$ are isomorphic.

A {\em congruence} on the labeled graph $\Gamma$ is an equivalence relation $\tau$ on $V(\Gamma)$ for which whenever $e_1, e_2 \in E(\Gamma)$ are such that $\alpha(e_1) \mathrel{\tau} \alpha(e_2)$ and $l(e_1)=l(e_2)$, they also satisfy $\omega(e_1) \mathrel{\tau} \omega(e_2)$. We can then define the quotient $\Gamma/\tau$ with vertex set $V(\Gamma)/\tau$ and for every edge $e$ of $\Gamma$, an edge $e\tau$ from $\alpha(e)\tau$ to $\omega(e)\tau$ with label $l(e)$.

A {\em congruence} on the automaton $\A = (\Gamma, I, T)$ is a congruence on $\Gamma$. We define the quotient $\A/\tau = (\Gamma/\tau,I\tau,T\tau)$ by taking $I\tau = \{ i\tau \mid i \in I\}$ and $T\tau = \{ t\tau \mid t \in T\}$.

A {\em pushdown automaton} (abbreviated to {\em pda}) is, informally speaking, a finite state automaton with a stack. A stack is a data type that contains a sequence of elements, that is, a word, and has two operations: 
\begin{itemize}
\item `popping': taking the first letter off the word, and
\item `pushing': concatenating a new word to the beginning of the word. 
\end{itemize}
A stack is initialized to contain a single letter called the initial stack symbol.

Formally a pushdown automaton is a structure of the form $\A = (A,S,Z,Q,q,T,\delta)$, where
\begin{itemize}
\item
$A$, $X$ and $Q$ are finite sets (alphabet, stack alphabet and states, respectively);
\item
$Z \in X$ and $q \in Q$ (initial stack symbol and initial state);
\item
$T \subseteq Q$ (terminal states);
\item
$\delta$ is a finite subset of $Q \times (A \cup \{ \epsilon \}) \times X \times Q \times X^*$ (set of transitions).
\end{itemize}

A transition $(p,a,x,q,u) \in \delta$ is also denoted by 
$$p \xrightarrow{a,\ x \to u} q.$$ It
has the meaning that in the state $p$, on the input $a \in A \cup \{ \epsilon\}$, with $x \in X$ as topmost stack symbol, the pda can read $a$, pop $x$, push $u \in X^\ast$, and change to the state $q$. (Reading $\epsilon$ from the input means leaving it untouched.)

{A \emph{configuration} of $\A$ is an element of $Q \times X^\ast$, with the first component describing the state, the second the content of the stack. Transitions of $\A$ give rise to transitions between configurations in a natural way. A transition from the configuration $(p, w)$ to $(q, u)$ on input $a \in A \cup \{\epsilon\}$ is denoted by $(p, w) \vdash_a (q, u)$. Note that the initial configuration of $\A$ is $(q,Z)$.}

We define the language accepted by $\A$ to be the language $L(\A)$ consisting of all $w \in A^*$ admitting a factorization of the form $w = w_1\ldots w_n$ $(w_j \in A \cup \{ \epsilon \}$) satisfying the following property: there exist $q_1,\ldots,q_{n-1} \in Q$, $q_n \in T$, $u_1, \ldots, u_n \in X^\ast$ and a series of transitions
$$(q, Z) \vdash_{w_1} (q_1,u_1) \vdash_{w_2} \cdots \vdash_{w_n} (q_n, u_n).$$

A language $L \subseteq A^*$ is {\em context-free} if it is the language accepted by some pushdown automaton.

\subsection{Inverse monoids}

An {\em inverse monoid} is a monoid $M$ with the
property that for each $s \in M$ there exists a unique element
$s^{-1} \in M$ (the inverse of $s$) such that $s =ss^{-1}s$ and
$s^{-1} = s^{-1}ss^{-1}$. We refer the reader to \cite{law} for an introduction to the subject.
The axioms imply $(s^{-1})^{-1}=s$ and $(st)^{-1}=t^{-1}s^{-1}$ for any $s,t \in M$, but in contrast with groups, $ss^{-1}$ and $s^{-1}s$ can be different from the identity or from one another. They are however always idempotent, moreover, since an idempotent of an inverse monoid is always its own inverse, every idempotent of $M$ is of this form. 

Idempotents play an important role in the theory of inverse monoids. For instance, notice that an inverse monoid is a group if and only if it has only one idempotent (the identity element).
In any inverse monoid $M$, idempotents commute with each other, and so the product of idempotents is again an idempotent. Therefore the set of idempotents forms a subsemigroup, moreover a subsemilattice in $M$. The semilattice structure induces a partial order on the idempotents in the usual way, and this can be extended to the whole inverse monoid as follows: for $s,t \in M$, we define $s \leq t$ if there exists an idempotent $e$ with $s=te$.
This is called the \emph{natural partial order} on $M$, and it is compatible with multiplication and taking inverses.

For studying the structure of inverse monoids, or semigroups in general, one of the most basic concepts is understanding when and how two elements can be multiplied into one another. This is encoded by the so-called Green's relations, from which we will only be needing one, the $\mathcal R$-relation. Two elements $s,t \in M$ are $\mathcal R$-related if they can be right multiplied into each other, that is if there exist elements $u,v \in M$ such that $su=t$ and $tv=s$. This is an equivalence relation. In inverse monoids, this turns out to be equivalent to saying $ss^{-1}=tt^{-1}$. The $\mathcal R$-class of an element $s$ is denoted by $R_s$. Notice that $ss^{-1}$ is always $\mathcal R$-related to $s$, so the idempotent $ss^{-1}$ is the unique idempotent in $R_s$. It may also be shown that if $s \mathrel{\mathcal R} t$ and $su=t$, then we have $tu^{-1}=s$.

Notice that in any Cayley graph of $M$, there is a path from $s \in M$ to $t \in M$ if and only if there exists $u \in M$ such that $su=t$, therefore the $\mathcal R$-classes correspond to the strongly connected components. If $a \in M$ is such that $s \mathrel{\mathcal R} sa$, then by our previous assertion, we have $saa^{-1}=s$. So within the strongly connected components, edges occur in inverse pairs (though this is not generally true in the Cayley graph). The strongly connected components of the Cayley graph are called Sch\"utzenberger graphs, and as they are symmetric graphs, they are metric spaces. However we would like to point out that in a Sch\"utzenberger graph, some generators may not label any edges at a given vertex. It may also happen that they have no non-trivial automorphisms -- this is the case for instance in the Sch\"utzenberger graph of $1$ in the example shown in Figure \ref{fig:bicyclic}.

Inverse monoids form a variety (in the sense of universal algebra, in signature $(2,1,0)$), hence free inverse monoids exist. The free inverse monoid on a set $A$, denoted by $\FIM(A)$, is obtained as a quotient of the free monoid $(A \cup A^{-1})^\ast$ equipped with the involution $^{-1}$ defined by $(a_1^{\epsilon_1} \ldots a_n^{\epsilon_n})^{-1}=a_n^{-\epsilon_n}\ldots a_1^{-\epsilon_1}$ for $a_i \in A, \epsilon_i=\pm 1$. The appropriate congruence is the so-called Wagner congruence denoted by $\rho$, and is generated by pairs of the form $ww^{-1}w \sim w$, and $w_1w_1^{-1}w_2w_2^{-1} \sim w_2w_2^{-1}w_1w_1^{-1}$, where $w, w_1, w_2 \in (A \cup A^{-1})^\ast$.

The obtained equivalence classes do not have a `reduced' normal form like in the case of free groups. Nevertheless there is a way to represent classes by graphs using a description due to Munn \cite{munn}. For any word $w \in (A \cup A^{-1})^\ast$, its \textit{Munn tree}
is the edge-labeled tree $\MT(w)$  traced out by $w$ in the Cayley graph of the free group ${\rm FG}(A)$ starting from the identity $1_{\FG(A)}$, with the initial vertex $1_{\FG(A)}$ and the terminal vertex $w_{\FG(A)}$ (the image of $w$ in the free group) marked. Munn proved that $u\rho=w\rho$ if and only if the triplets $(\MT(u), 1_{\FG(A)}, {u}_{\FG(A)})$ and $(\MT(w), 1_{\FG(A)}, {w}_{\FG(A)})$ are isomorphic as birooted edge-labeled graphs, or in other words, as automata. The language of the automaton $(\MT(w), 1_{\FG(A)}, {w}_{\FG(A)})$ therefore clearly contains the $\rho$-class of $w$, but also more than that: for instance notice that if $a,b \in A$, then $a$ is accepted by the automaton corresponding to $\MT(abb^{-1})$, but $a\rho \neq (abb^{-1})\rho$. What is true in general is that the language of $(\MT(w), 1_{\FG(A)}, {w}_{\FG(A)})$ consists of all words $u$ for which $u\rho \geq w\rho$ in the natural partial order on the elements in the free inverse monoid $\FIM(A)$

The Munn tree $\MT(w)$ is isomorphic to the Sch\"utzenberger graph containing $w\rho$, with its initial vertex corresponding to $(ww^{-1})\rho$, and its terminal vertex corresponding to $w\rho$. Stephen extended Munn's
results \cite{Ste1} to study the word problem in arbitrary
inverse monoids, with Munn trees replaced by Sch\"utzenberger graphs.
Any inverse monoid $M$ generated by $A$ (as an inverse monoid) is the quotient of the free monoid $(A \cup A^{-1})^\ast$ with involution by a congruence $\tau$ (which contains $\rho$). The word problem for $M$ is the problem of deciding if two arbitrary words $u, v \in (A \cup A^{-1})^\ast$ fall into the same $\tau$-class, that is, if they represent equal elements of $M$.

The Sch\"utzenberger graph $\S(w)$ of a word $w \in (A
\cup A^{-1})^*$ is the Sch\"utzenberger graph of $M$ containing $w\tau$. The {\em Sch\"utzenberger automaton} ${\mathcal A}(w)$ of $w$ is the automaton  $(\S(w),(ww^{-1})\tau,w\tau)$.  In his paper \cite{Ste1}, Stephen proves the
following result.

\bt
\label{Ste}
Let $M=\langle A \rangle$ be an inverse monoid. Then

(a) For each word $u \in (A \cup A^{-1})^*$, the  language accepted by the \\
Sch\"utzenberger automaton ${\mathcal A}(u)$ is the set of all
words $w \in (A \cup A^{-1})^*$ such that $u\tau \leq w\tau$ in the natural partial
order on $M$,

(b) $u\tau = w\tau$ in $M$ iff $u \in L({\mathcal A}(w))$ and $w \in
L({\mathcal A}(u))$,

(c) the word problem for $M$ is decidable if and only if there is an algorithm
which takes any two words $u,w
\in (A \cup A^{-1})^*$ as input, and decides if $u \in L({\mathcal A}(w))$.
\et

More can be said when $M$ is given by a presentation. A presentation $M=\invp\langle A\mid R \rangle$ , where $R \subseteq (A \cup A^{-1})^\ast \times (A \cup A^{-1})^\ast$, defines the inverse monoid $M$ as the quotient of $\FIM(A)$ by the congruence generated by $R$, or alternatively, as the quotient $(A \cup A^{-1})^\ast$ by the congruence generated by $\tau=\rho \cup R$. Notice that since every inverse monoid is a quotient of some free inverse monoid, every inverse monoid has a presentation.

Stephen also proved that $\S(w)$, and equivalently, $\A(w)$, can be obtained as a direct limit of successive expansions and foldings as follows.
Let $\Gamma$ be a digraph labeled over $A \cup A^{-1}$ in a way that edges occur in inverse pairs. We define the following operations:
\begin{itemize}
\item $P$-expansion: if $s=t$ or $t=s$ is a relation in $R$, and for some $u,v \in V(\Gamma)$, there exists a path $p_s$ labeled by $s$ with $\alpha(p_s)=u$ and $\omega(p_s)=v$, but there is no such path between $u$ and $v$ labeled by $t$, then adjoin a simple path $p_t$ labeled by $t$ to $\Gamma$ with $\alpha(p_t)=u$ and $\omega(p_t)=v$, and all internal vertices of $p_t$ disjoint from $\Gamma$. (Here we add all inverse pairs of edges of $p_t$ to $\Gamma$ as well to obtain a symmetric graph.)
\item edge folding: if there are edges $e_1$ and $e_2$ of the same label with  common initial or common terminal vertices, identify these edges.
\end{itemize}

By Stephen \cite{Ste1}, these operations are confluent. Starting with any graph $\Gamma$, the set of all graphs obtained by applying successive
$P$-expansions and edge foldings forms a directed system in the
category of $A$-labeled graphs. The direct limit of this system is denoted by $\exp(\Gamma)$.

{We define the full $P$-expansion of $\Gamma$ to be the graph obtained by performing all possible $P$-expansions of $\Gamma$. We emphasize that this only involves $P$-expansions which can be performed on $\Gamma$ itself, not involving any newly added paths. The iteration of edge-foldings to obtain a deterministic graph is called determinization. Given a graph $\Gamma$, let $\exp_1(\Gamma)$ denote the determinization of the full $P$-expansion of $\Gamma$. For $i \geq 2$, let $\exp_i(\Gamma)= \exp_1(\exp_{i-1}(\Gamma))$.  Then $\exp(\Gamma)$ is also the direct limit of the directed system formed by $\{\exp_i(\Gamma): i \in \mathbb N\}$. For convenience, we extend our notation by $\exp_0(\Gamma):=\Gamma$. 

Stephen proved the following in \cite{Ste1}:

\bt
For any inverse monoid presentation $\invp\langle A \mid R\rangle$ and any word $w \in (A \cup A^{-1})^\ast$, $\exp(\MT(w))=\S(w)$.
Furthermore, the images of $1_{\FG(A)}$, and $w_{\FG(A)}$ in the limit graph $\exp(\MT(w))$ correspond to $(ww^{-1})\tau$ and $w\tau$ respectively.
\et

As an example, let us consider the \textit{bicyclic monoid}, which is the most famous inverse monoid that is not a group. It can be defined by the presentation $\invp\langle a \mid aa^{-1}=1\rangle$. Figure \ref{fig:bicyclic} shows a finite part of the Cayley graph of the bicyclic monoid. The Sch\"utzenberger graphs are colored black, the other edges are grey.
Stephen's algorithm to build $\S(1)$ for instance would start with the single vertex $1$, then the first expansion would attach a loop labeled by $aa^{-1}$, the second expansion a second loop based at the vertex $a$ labeled by $aa^{-1}$, and so on. In this particular case, $\exp_i(\MT(1))$ embeds in $\exp_{i+1}(\MT(1))$ for any $i \in \mathbb N$, however this is not true in general.

\begin{center}
\begin{figure}[h]
\includegraphics[scale=0.65]{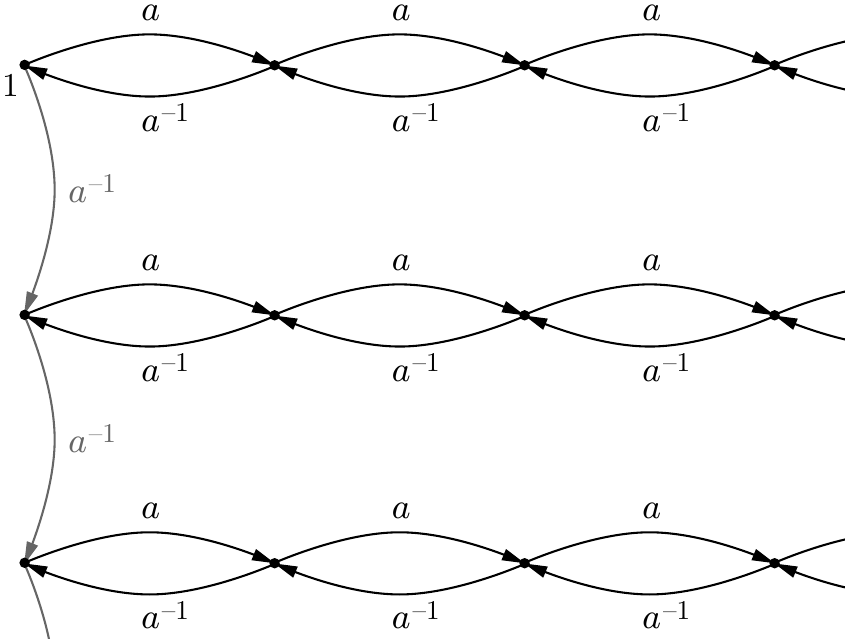}
\caption{The Cayley graph of the bicyclic monoid}
\label{fig:bicyclic}
\end{figure}
\end{center}

An approximate automaton $\A$ of $\A(w)$ is an inverse automaton labeled over $A \cup A^{-1}$ such that $w \in L(\A) \subseteq L(\A(w))$. An approximate graph is the underlying (labeled) graph of an approximate automaton. Any deterministic graph obtained by expansions and foldings from $\MT(w)$ is an approximate graph of $\S(w)$. {Conversely, every finite subgraph of $\S(w)$ is a subgraph of $\exp_k(\MT(w))$ for some $k$, this follows from \cite[Thm 5.11]{Ste1}.}



Every inverse monoid has a greatest group homomorphic image which is obtained by collapsing all idempotents to the identity. For instance the greatest group homomorphic image of $\FIM(A)$ is the free group $\FG(A)$, and the greatest group homomorphic image of the bicyclic monoid is $\mathbb Z$. If the inverse monoid is given by a presentation $M=\invp\langle A \mid R\rangle$, then this is the group given by the same presentation regarded as a group presentation $G=\Gpres{A}{R}$. Denoting the image of $w \in (A\cup A^{-1})^\ast$ in $G$ by $w_G$, the canonical homomorphism $\sigma:  M \rightarrow G$ maps $w\tau$  to $w_G$ for any $w \in (A\cup A^{-1})^\ast$. 

This defines a map $\overline{\sigma}$ from the union of Sch\"utzenberger graphs of $M$ to the Cayley graph of $G$ in the obvious way, and this is an edge-labeled graph morphism. In the case of  $\FIM(A)$ for instance, this is the inclusion map of Munn trees into the Cayley graph of $\FG(A)$. In case of the bicyclic monoid, $\overline{\sigma}$ embeds each Sch\"utzenberger graph into the Cayley graph of $\mathbb Z$. However, in general, $\overline{\sigma}$ is not injective on individual Sch\"utzenberger graphs. For instance any inverse monoid with a zero has a trivial greatest group homomorphic image, but can very well have non-trivial Sch\"utzenberger graphs.

An inverse monoid is called \emph{$E$-unitary} if $\sigma^{-1}(1_G)$ is exactly the set of idempotents of the monoid, and this property is equivalent to $\overline{\sigma}$ being injective on each Sch\"utzenberger graph. When the inverse monoid is given by a special presentation, more is true: 

\begin{Lemma}[\cite{Ste2}, Lemma~3.5]\label{lem:full:sibgraph}  
Let $M= \Ipres{A}{w_1=1, \ldots, w_k=1}$, and let $G=\Gpres {A}{w_1=1, \ldots, w_k=1}$ be its greatest group homomorphic image. If $M$ is $E$-unitary, then for all $w \in (A \cup A^{-1})^*$ the Sch\"{u}tzenberger graph $\mathcal{S}(w)$ is embedded into the Cayley graph of $G$ by $\overline{\sigma}$ as an induced subgraph.         
\end{Lemma}

\section{Equivalent characterizations of tree-like graphs}

Recall that a \emph{quasi-isometry} between metric spaces $(X, d_X)$ and $(Y, d_Y)$ is a map $\psi: X \to Y$ for which there exist positive constants $k,c, K,C$ such that for all $x_1, x_2 \in X$,
$$kd_X(x_1,x_2)-c \leq d_Y(\psi(x_1),\psi(x_2)) \leq Kd_X(x_1,x_2)+C,$$
furthermore there exists a constant $m$ such that for all $y \in Y$, there exists $x \in X$ with $d(y, \psi(x)) \leq m$.
We say that a symmetric graph $\Gamma$ is \emph{tree-like} if it is quasi-isometric to a tree.
In this section, we describe a few equivalent characterizations of tree-like graphs, and assume all graphs considered to be symmetric.

For any graph $\Gamma$ and any partition $\P$ of $V(\Gamma)$, we define the undirected graph $\Gamma/{\P}$ by $V(\Gamma/{\P})=V(\Gamma)/\P$, and for any blocks $S_1, S_2\in \P$, there is an edge between $S_1$ and $S_2$ if and only if $S_1 \neq S_2$, and there exist vertices $u_i \in S_i$ $(i=1,2)$ such that there is an edge between $u_1$ and $u_2$ is $\Gamma$. Note that $\Gamma/{\P}$ is, by definition, simple.
A \emph{strong tree decomposition} of a graph $\Gamma$ is a partition $\P$ of $V(\Gamma)$ such that $\Gamma/{\P}$ is an (undirected) tree. We call $\sup\{\operatorname{diam}(S): S \in \mathcal P\}$ the \emph{width} of the decomposition (which may or may not be finite).
 
{A {\em geodesic polygon} $G$ of a symmetric graph $\Gamma$ is a a sequence $p_1,\ldots ,p_n$ $(n \geq 2)$ of geodesics, where $\alpha(p_{i+1})=\omega(p_i)$ $(1 \leq i \leq {n-1})$, and $\alpha(p_1)=\omega(p_n)$. These geodesics form the sides of the polygon.  If $\delta \geq 0$, we say that $G$ is $\delta$-{\em thin} if for all $i \in \{ 1,\ldots, n\}$, for all vertices $y$ in $p_i$, there exists a vertex $z$ in $\bigcup_{j \neq i} p_j$ such that $d(y,z) \leq \delta$.}

We say that $\Gamma$ is {\em polygon $\delta$-hyperbolic} if all geodesic polygons of $\Gamma$ are $\delta$-thin. We say that $\Gamma$ is {\em polygon hyperbolic} if it is polygon $\delta$-hyperbolic for some $\delta \geq 0$. This is a special case of Gromov's notion of hyperbolicity, which $\Gamma$ satisfies if there is a constant $\delta$ such that every geodesic triangle of $\Gamma$ is $\delta$-thin (in which case $\Gamma$ is $\delta$-hyperbolic). This condition is known as the Rips condition.

Gromov originally defined hyperbolicity using what is now called the Gromov product. If $x,y,z \in V(\Gamma)$, then the Gromov product of $x$ and $y$ at $z$ is $(x|y)_z=\frac{1}{2}(d(z,x)+d(z,y)-d(x,y))$. Then $\Gamma$ is $\delta$-hyperbolic if for any $x_1, x_2, y,z \in V(\Gamma)$ we have 
$$(x_1|x_2)_z \geq \min\{(x_1|y)_z,  (x_{2}|y)_z)\}-\delta.$$

Let $x_0,x \in V(\Gamma)$. The {\em cone} of $x$ with respect to the basepoint $x_0$ is defined as
$$C(x_0,x) = \langle y \in V(\Gamma) \mid d(x_0,y) = d(x_0,x) + d(x,y) \rangle.$$
This is the same as saying that there exists a geodesic between $x_0$ and $y$ containing $x$.

\bp
\label{prop_tree}
For any symmetric graph $\Gamma$, the following conditions are equivalent:
\begin{enumerate}
	\item $\Gamma$ is quasi-isometric to a tree,
	\item there exists a constant $m >1$ and a strong tree decomposition $\P$ of $\Gamma$ with width at most $m$.
	\item there exists a constant $m >1$ such that for any $x,y,z \in V(\Gamma)$ and any path $p$ from $x$ to $y$,
	$$d(z,p) \leq \frac{1}{2}(d(z,x)+d(z,y)-d(x,y))+m,$$
	\item there exists some constant $k$ such that for any $x_0, \ldots, x_n, z \in V(\Gamma)$,
	$$(x_0|x_n)_z \geq \min\{(x_0|x_1)_z, \ldots, (x_{n-1}|x_n)_z)\}-k,$$
	\item $\Gamma$ is polygon hyperbolic,
	{\item there exists a constant $\delta$ such that if $x_0 \in V(\Gamma) \setminus D_{\delta}(x)$ and $y \in V(C(x_0,x)) $, then there is no path in $\Gamma \setminus D_{\delta}(x)$ connecting $x_0$ to $y$.}
\end{enumerate}
\ep

\proof
By \cite[Theorem 4.7]{Ant}, conditions (1)--(3) are equivalent to each other, whereas by \cite[Theorem 3.2]{Silva}, so are conditions (4)--(5). We complete the proof by showing (4) $\Rightarrow$ (3), (1) $\Rightarrow$ (5), and (5) $\Leftrightarrow$ (6).

(4) $\Rightarrow$ (3): let $x,y,z \in V(\Gamma)$, and let $p$ be a path from $x$ to $y$. Let $x=x_0, x_1, \ldots, x_{n-1}, x_n=y$ be the sequence of vertices $p$ traverses. Then for any $i$, $d(x_i,x_{i+1}) \leq 1$, hence 
$$(x_i|x_{i+1})_z \geq \frac{1}{2}(d(z,x_i)+d(z,x_{i+1})-1)\geq\min\{d(z,x_i), d(z,x_{i+1})\}-1.$$
It follows that 
$$(x|y)_z+k \geq \min\{(x_0|x_1)_z, \ldots, (x_{n-1}|x_n)_z)\} \geq \min\{d(z,x_0), \ldots, d(z,x_n)\}-1=d(z,p)-1,$$
and hence by $(4)$, $\frac{1}{2}(d(z,x)+d(z,y)-d(x,y))+k+1=(x|y)_z+k+1\geq d(z,p)$, and $m=1+k$ satisfies the conditions of $(3)$.

(1) $\Rightarrow$ (5): suppose $\psi: \Gamma \to T$ is a quasi-isometry from $\Gamma$ to some tree $T$, with constants $k,c, K, C$. The key observation is that there exists a constant $\delta$ such that if $[u,v]$ is a geodesic in $\Gamma$, then the path $\psi([u,v])$  and any geodesic $[\psi(u),\psi(v)]$ are in each other's $\delta$-neighborhood. This statement holds in the more general case of quasi-isometries into hyperbolic spaces, and is proved for instance in \cite{BH} (Chapter III.H).

Let $G =  p_1,p_2, \ldots ,p_n=[x_0,x_1],[x_1,x_2],\ldots ,[x_n,x_0]$ be a geodesic polygon in $\Gamma$. We claim that $\psi(G)$ is $2\delta$-thin. Indeed, let $u \in \psi([x_i,x_{i+1}])$. Then there exists some vertex $u'$ in some geodesic of the form $[\psi(x_i), \psi(x_{i+1})]$ such that $d(u,u')\leq \delta$. Consider a geodesic polygon of the form
$G'=[\psi(x_0),\psi(x_1)],[\psi(x_{1}),\psi(x_2)],\ldots ,[\psi(x_n),\psi(x_0)]$, where $[\psi(x_i), \psi(x_{i+1})]$ is the geodesic containing $u'$.
As trees are polygon $0$-hyperbolic, $u'$ is contained in a different side of $G'$ as well, say, $[\psi(x_j),\psi(x_{j+1})]$ $(j \neq i)$. Let $u''$ be a vertex of $\psi([x_j, x_{j+1}])$ such that $d(u',u'') \leq \delta$. Then $d(u,u'') \leq 2\delta$, and $\psi(G)$ is indeed $2\delta$-thin.

Now let $x \in [x_{i}, x_{i+1}]$. There exists some vertex $y \in [x_j, x_{j+1}]$, $j \neq i$ such that $d(\psi(x),\psi(y))\leq 2\delta$, hence $d(x,y)\leq \frac{1}{k}(d(\psi(x),\psi(y))+c) \leq \frac{2\delta+c}{k}.$ Therefore $G$ is $\frac{2\delta+c}{k}$-thin, and $\Gamma$ is polygon hyperbolic.

{(5) $\Rw$ (6). Suppose $\Gamma$ is polygon $\delta$-hyperbolic, and let $x_0 \in V(\Gamma) \setminus D_{\delta}(x)$ and $y \in V(C(x_0,x))$. Suppose that there exists a path $e_1\ldots e_n$ in $\Gamma \setminus D_{\delta}(x)$ connecting $x_0$ to $y$. Consider a geodesic $[y,x_0]$ in $\Gamma$ containing $x$ (its existence follows from $y \in C(x_0,x)$) and consider the geodesic polygon
$G =  e_1,\ldots ,e_n ,[y,x_0]$.
Since $x \in [y,x_0]$ and $\Gamma$ is polygon $\delta$-hyperbolic, there exists some $z \in e_1\ldots e_n$ such that $d(x,z) \leq \delta$, contradicting $e_1\ldots e_n \subseteq \Gamma \setminus D_{\delta}(x)$. Therefore there is no path in $\Gamma \setminus D_{\delta}(x)$ connecting $x_0$ to $y$.

(6) $\Rw$ (5). We claim that $\Gamma$ is polygon $\delta$-hyperbolic with the $\delta$ of condition (6). Suppose that this is not the case. Then we may assume that there exists a geodesic polygon of the form $p_1, \ldots, p_n$ and $x \in p_n$ such that $p_1, \ldots ,p_{n-1} \subseteq \Gamma \setminus D_{\delta}(x)$. Let $x_0=\alpha(p_1)=\omega(p_n)$, and let $y=\omega (p_{n-1}) = \alpha (p_n)$.
Then $p_1 \ldots p_{n-1}$ is a path from $x_0$ to $y$ in $V(\Gamma) \setminus D_{\delta}(x)$, and $y \in C(x_0,x)$, therefore (6) fails.}
\qed

Note that by $(1)$, being tree-like is clearly a quasi-isometry invariant.

\section{Tree-like Sch\"utzenberger graphs}
\label{Sec.treelike}

In this section, we consider finitely presented inverse monoids with tree-like Sch\"utzenberger graphs, and we prove that they share some of the nice algorithmic properties of hyperbolic and virtually free groups.

Let $M={\rm Inv}\langle A\rangle$ be a finitely generated inverse monoid. If for any word $w \in (A \cup A\inv)^*$, the Sch\"utzenberger graph $\S(w)$ is tree-like, we say that $M$ is tree-like. 
{This definition does not depend on $A$, as changing the system of generators yields quasi-isometries between the respective Sch\"utzenberger graphs \cite{BM}. (We remark that \cite{BM} proves something more general, namely that the Cayley graphs of a monoid with respect to two different systems of generators are quasi-isometric as semimetric spaces; the quasi-isometries of the Sch\"utzenberger graphs are the restrictions of this quasi-isometry to the $\mathcal R$-classes.)
}

{If an inverse monoid happens to be a group, then it has a unique Sch\"utzenberger graph: the Cayley graph, and it is tree-like if and only if it is a virtually free group. Other examples include 
free inverse monoids (where Sch\"utzenberger graphs are not only trees, but finite trees), and presentations of the form ${\rm Inv}\langle A \mid R \rangle$, where each word occurring in $R$ is freely reducible to $1$: here Sch\"utzenberger graphs are trees, see \cite[Lemma 1.6]{MM}.

All of the theorems that we prove in this section are based on the same observation that can be roughly described as follows: take an induced, connected subgraph $X$ of $\S(w)$ that does not intersect the path labeled by $w$, starting at $(ww^{-1})\tau$. Then the $K$-neighborhood of $X$ (where $K$ is a constant depending on the presentation) completely determines those connected components of $\S(w) \setminus X$ which do not contain $(ww^{-1})\tau$. We begin the section by the precise formulation and proof of the above claim.

Let ${\rm Inv} \langle A \mid r_1=s_1,\ldots, r_m=s_m \rangle$ be a finite inverse monoid presentation, and $w \in (A \cup A\inv)^*$, and denote $(ww\inv)\tau$ by $x_0$. 
Put $K = \max\{2, |r_1|,\ldots,|r_m|,|s_1|,\ldots,|s_m|\}$.

Given any approximate graph $S$ of $\S(w)$, and a connected, induced subgraph $X$ of $S$ with $x_0 \notin X$, let $\x0{X}$ denote the connected component of $S \setminus X$ containing $x_0$, and $\cc0{X}$ denote the complement of $\x0{X}$  in  $S \setminus X$. 



Recall that $\S(w)=\Exp(\MT(w))$, denote the image of $\MT(w)$ in $\S(w)$ by $M$. Suppose $X$ is an induced, connected subgraph of $\S(w)$ such that $X \cap M =\emptyset$, and let $g$ be a path from $M$ to $X^{+K}$ in $\S(w)$. Of course $\S(w)=\Exp(M \cup g \cup X^{+K})$. Let $x_0=(ww^{-1})\tau$, and consider the subgraphs $\ck{k}{X}$ and $\cck{k}{X}$ of  $\S(w)$. Note that $\ck{k}{X}$ contains the image of $M$, since $X$ does not intersect $M$.

Consider the graph $\Exp(X^{+K})$, observe that $X^{+K}$ is a subgraph here. Note that for any $P$-expansion or folding performed when building $\Exp(X^{+K})$, an identical $P$-expansion or folding may be performed when building $\S(w)=\Exp(M \cup g \cup X^{+K})$, and hence there is a unique labeled graph morphism $\varphi \colon \Exp(X^{+K}) \to \S(w)$ which fixes $X^{+K}$. 
As $X^{+K}$ is an induced subgraph of $S(w)$, the existence of $\varphi$ immediately implies that it is an induced subgraph of $\exp(X^{+K})$ as well, moreover, we claim it is the $K$-neighborhood of $X$ in $\exp(X^{+K})$. Denote the latter subgraph by $\tilde X^{+K}$. 
Clearly $X^{+K} \subseteq \tilde X^{+K}$. For the opposite containment, let $v \in V(\tilde X^{+K})$, and let $p$ be a geodesic path in $\exp(X^{+K})$ from $X$ to $v$. Note that since $\varphi$ fixes $X^{+K}$, we have $\varphi(\alpha(p))=\alpha(p)$. Furthermore, $\varphi(p)$ is a path of length at most $K$ in $\S(w)$ from $\varphi(X)=X$ to $\varphi(v)$, therefore $\varphi(p)$ lies completely in $X^{+K}$, (in particular, $\varphi(v) \in X^{+K}$). 
Since $X^{+K}\subseteq  \exp(X^{+K})$, we have $\varphi(p) \in \exp(X^{+K})$, in fact, $\varphi(p)$ is a path in $\exp(X^{+K})$ from $\alpha(p)$ labeled by $l(p)$. Since $\exp(X^{+K})$ is deterministic, this implies $p=\varphi(p)$, in particular, $v=\varphi(v) \in X^{+K}$.
This proves the claim, and allows us to use the notation $X^{+r}$ for $r\leq K$ unambiguously.

\bl
\label{explem}
The above morphism $\varphi$ restricted to $\varphi^{-1}(\cck{k}{X})$ is an isomorphism between $\varphi^{-1}(\cck{k}{X})$ and $\cck{k}{X}$, and hence $\varphi^{-1}$ embeds $\cck{k}{X}$ into $\Exp(X^{+K})$ as a subgraph.
\el

\proof
The idea of the proof is that all the $P$-expansions and foldings which yield $\cck{k}{X}$ are also performed in the construction of $\Exp(X^{+K})$.
Recall that a $P$-expansion means adjoining a new simple path $p_s$ labeled by $s$ parallel to an existing path $p_t$ labeled by $t$, where $s=t$ is a relation. We say that an edge or vertex of $\S(w)$ was \textit{involved} in such a given $P$-expansion if it is an edge or vertex of the image of the respective paths $p_s$ or $p_t$. In addition, if the edge or vertex is in the image of $p_t$, we say that the $P$-expansion was \textit{dependent} on that edge or vertex.
Similarly, we say an edge or vertex of $\S(w)$ was involved in a given folding operation identifying a pair of edges if it is the image of these edges or one of their endpoints, and in this case the folding is also dependent on this edge or vertex. 

Suppose $\gamma_1, \ldots ,\gamma_n$ is a sequence of $P$-expansions and foldings performed in the construction of $\S(w)$ as $\Exp(M \cup g \cup X^{+K})$, furthermore suppose it is a sequence of operations performable on $M \cup g \cup X^{+K}$ in this order, without the need to perform any other operations.
Let $P_i$ $(1 \leq i \leq n)$ be the subgraph of $\S(w)$ consisting of all the edges and vertices which were involved in $\gamma_i$, and let $D_i$ be the subgraph of $\S(w)$ consisting of all the edges and vertices on which $\gamma_i$ is dependent. Observe that $D_i \subseteq \ck{k}{X} \cup P_1 \cup \cdots \cup P_{i-1}$ must hold.
Also note that $\operatorname{diam}(P_i) \leq K$ for any $i$, and hence if $\cck{k}{X} \cap P_i \neq \emptyset$, then $P_i \subseteq X^{+K} \cup \cck{k}{X}$, and similarly if $\ck{k}{X} \cap P_i \neq \emptyset$, then $P_i \subseteq X^{+K} \cup \ck{k}{X}$.

Let $j_1, \ldots, j_m$ be the sequence of indices $i$ for which $P_i \cap \cck{k}{X} \neq \emptyset$, and consider the sequence $\gamma_{j_1}, \ldots, \gamma_{j_m}$. We claim that an identical sequence of $P$-expansions and foldings are performable on $X^{+K}$ (in this order, without the need to perform any other operations).

Indeed, note that $P_1 \cup \cdots \cup P_{j_1-1} \subseteq X^{+K} \cup \ck{k}{X}$, hence $D_{j_1} \subseteq X^{+K} \cup \ck{k}{X}$, but also $D_{j_1} \subseteq P_{j_1} \subseteq  X^{+K} \cup \cck{k}{X}$, hence  $D_{j_1} \subseteq X^{+K}$, and therefore an operation identical to $\gamma_{j_1}$ can be performed on $X^{+K}$. Observe that $\gamma_{j_1+1}, \ldots, \gamma_{j_2-1}$ do not change the new subgraph arising from the operation $\gamma_{j_1}$, as $P_i \subseteq X^{+K} \cup \ck{k}{X}$ for all $j_1 <i <j_2$.

By induction, suppose a sequence of operations identical to $\gamma_{j_1}, \ldots \gamma_{j_l}$ are performable on $X^{+K}$, where $1\leq l< m$, and suppose $\gamma_{j_l+1}, \ldots, \gamma_{j_{l+1}-1}$ do not alter the subgraph arising from the operations $\gamma_{j_1}, \ldots \gamma_{j_l}$. Then 
$$D_{j_{l+1}} \subseteq  \bigcup_{i=1}^{j_{l+1}-1}P_i \cap (X^{+K} \cup \cck{k}{X})\subseteq\left(\bigcup_{i=1}^{l}P_{j_i} \cap (X^{+K} \cup \cck{k}{X})\right) \cup X^{+K} \subseteq \left(\bigcup_{i=1}^{l}P_{j_i}\right) \cup X^{+K},$$
an therefore an operation identical to $\gamma_{j_{l+1}}$ can be performed after $\gamma_{j_1}, \ldots \gamma_{j_l}$ on $X^{+K}$. Again, for any $j_{l+1} <i < j_{l+2}$, we have $P_i \subseteq X^{+K} \cup \ck{k}{X}$, hence $\gamma_{j_{l+1}+1}, \ldots, \gamma_{j_{l+2}-1}$ do not alter any of the changes arising from any of the operations $\gamma_{j_1}, \ldots \gamma_{j_{l+1}}$. This proves that a sequence of operations identical to $\gamma_{j_1}, \ldots \gamma_{j_m}$ are indeed performable on $X^{+K}$, and therefore is performed during the construction of $\exp(X^{+K})$.

To prove that $\cck{k}{X}$ is contained in the image of $\varphi$, suppose $e$ is an edge of $\cck{k}{X}$. Take a minimal sequence $\gamma_1, \gamma_2, \ldots, \gamma_n$ of $P$-expansions and foldings of $M \cup g \cup X^{+K}$ which yield $e$.  By the minimality of the sequence, $\gamma_n$ is the $P$-expansion that yields $e$. Take the subsequence $\gamma_{j_1}, \ldots, \gamma_{j_k}$ defined as above, note that since $P_n \cap  \cck{k}{X} \neq \emptyset$, we have $j_k=n$. Therefore an operation identical to $\gamma_n$ is performed in the construction of $\exp(X^{+K})$, and this yields a preimage of $e$.

For the injectivity of the restriction of $\varphi$ to $\varphi^{-1}(\cck{k}{X})$, suppose $e_1, e_2$ are distinct edges of $\varphi^{-1}(\cck{k}{X}) \subseteq \exp(X^{+K})$. Sequences of operations identical to the ones that yield $e_1$ and $e_2$ in $\exp(X^{+K})$ are performable on $X^{+K}$ in $M \cup g \cup X^{+K}$, yielding edges $e_1'$ and $e_2'$ in some approximate graph of $S(w)$. As $\varphi$ is a labeled graph morphism fixing $X^{+K}$, it must be that the image of $e_i'$ in $\S(w)$ is $\varphi(e_i)$ $(i=1,2)$. Suppose $\varphi(e_1)=\varphi(e_2)$. Then there must be a folding $\gamma$ performed during the construction of $S(w)$ which identifies $e_1'$ and $e_2'$, and $\gamma$ involves the single edge $\varphi(e_1)=\varphi(e_2) \in \cck{k}{X}$. Hence
by the same argument as above, an operation identical to $\gamma$ is performed in the construction of $\exp(X^{+K})$, but this should have identified $e_1$ and $e_2$, which is a contradiction.
\qed

For the rest of the paper, we identify $\cck{k}{X}$ with its image under $\varphi^{-1}$ and consider it as a subgraph of $\Exp(X^{+K})$, in notation, $\cck{k}{X} \subseteq\Exp(X^{+K})$.

Define a \emph{coloring} on the vertices $X^{+K}$ the following way: let vertices of $X^{+K} \cap \cck{k}{X}$ be red, the rest blue. 

\bl
\label{color}
Denote the connected components of the red vertices in $\Exp (X^{+K}) \setminus X$ by $\Gamma^c$. Then $\varphi^{-1}(\cck{k}{X})=\Gamma^c$.
\el

\proof
Note it suffices to prove that the set of vertices of the two subgraphs are equal, as both $\Gamma^c$ and $\varphi^{-1}(\cck{k}{X})$ are induced subgraphs of $\Exp (X^{+K}) \setminus X$ by definition.

$V(\varphi^{-1}(\cck{k}{X}))\subseteq V(\Gamma^c)$:
Let $u \in V(\cck{k}{X})$, and take a red vertex $v \in X^{+K} \cap \cck{k}{X}$ in the connected component of $u$. Let $p$ be a path from $v$ to $u$ in $\cck{k}{X}$. Since $\varphi$ is a an isomorphism, $\varphi^{-1}(p)$ is a path from $\varphi^{-1}(v)=v$ to $\varphi^{-1}(u)$ in $\varphi^{-1}(\cck{k}{X}) \subseteq \exp(X^{+K}) \setminus X$, therefore $u$ is the same component of $\exp(X^{+K}) \setminus X$ as $v$.

$V(\varphi^{-1}(\cck{k}{X}))\supseteq V(\Gamma^c)$:
Let $u \in V(\Gamma^c)$, we prove $\varphi(u) \in \cck{k}{X}$ by induction on $d(u, X^{+K})$ in $\exp(X^{+K}) \setminus X$. 
If  $d(u, X^{+K})=0$, that is, $u \in X^{+K} \cap \cck{k}{X}$, then $\varphi(u)=u\in \cck{k}{X}$.
Suppose that the statement holds for all points $u'$ with $d(u', X^{+K})\leq k-1$, take a point $u \in \cck{k}{X}$ with $d(u, X^{+K})=k$, and let $q$ be a geodesic of length $k$ from $X^{+K}\cap \cck{k}{X}$ to $u$ in $\exp(X^{+K}) \setminus X$. In particular, $q$ must lie in $\Gamma^c$. Note that $\varphi(q)$ is a path in $S(w)$ from $X^{+K}\cap \cck{k}{X}$ to $\varphi(u)$, hence to see $\varphi(u) \in \cck{k}{X}$, we only need to show that $\varphi(q) \cap X =\emptyset$.

Let the sequence of vertices $q$ traverses be $u_0, \ldots, u_{k-1}, u_k=u$. Note that $u_1, \ldots, u_k \notin X^{+K} \cap \cck{k}{X}$ (otherwise $q$ would not be a geodesic), therefore $d(u_i, X) \geq K$ for all $i=0,\ldots,k$, that is, $u_i \notin X^{K-1}$ for all $i=0,1,\ldots,k$.
By the induction hypothesis, $\varphi(u_{i}) \in \cck{k}{X}$ for $i=0, \ldots, k-1$, furthermore, $\varphi(u_{i}) \notin X^{K-1}$ for $i=0, \ldots, k-1$, as each vertex in $X^{K-1} \cap \cck{k}{X}$ has only one preimage  --- itself --- by the injectivity of $\varphi|_{\varphi^{-1}(\cck{k}{X})}$. Since $K\geq 2$, and $d(\varphi(u_{n-1}), \varphi(u))=1$, this implies $\varphi(u) \notin X$, therefore $\varphi(u) \in \cck{k}{X}$.
\qed

Let $X, Y$ be subgraphs of $\S(w)$ which do not intersect $M$. If $X^{+K}$ and $Y^{+K}$ are isomorphic as vertex-colored graphs (with the red-blue coloring introduced above), we put $X^{+K} \cong_c Y^{+K}$.

\bl
\label{phrat1}
If
$X^{+K} \cong_c Y^{+K}$, then $X^{+K} \cup \cc0{X} \cong Y^{+K} \cup \cc0{Y}$.
\el

\proof
By Lemmas \ref{explem} and \ref{color}, $X^{+K} \cup \cc0{X} \subseteq \exp(X^{+K})$, moreover, $\cc0{X}$ consists of all the connected components of $\exp(X^{+K}) \setminus X$ containing red vertices. Since $\exp(X^{+K}) \cong_c \exp(Y^{+K})$ follows clearly from the assumption, this proves our statement.
\qed
}

\subsection{Geodesics}
In this subsection, we prove that if $\S(w)$ is tree-like and the presentation is finite, then the set of geodesics in $\S(w)$ starting at $x_0$ is rational.
We denote this set by ${\rm Geo}(w)$, more precisely, it is the set of all $u \in (A \cup A\inv)^*$ such that $(ww\inv)\tau \mapright{u} (ww\inv u)\tau$ is a path in $\S(w)$, and $u$ has minimum length among all words labelling paths from $(ww\inv)\tau$ to $(ww\inv u)\tau$ in $\S(w)$.


Suppose $A$ is finite. 
We begin by showing that being tree-like does not imply rational geodesics if $R$ is not finite. By Theorem \ref{phrat}, this means being finitely generated tree-like does not imply finite presentability, in contrast with virtually free or even hyperbolic groups.

\be
\label{irrtree}
Consider the inverse monoid presentation 
\beq
\label{irrtree1}
{\rm Inv}\langle a,b,c \mid aa\inv b^{n^2}cc\inv b^{-n^2} = aa\inv \; (n \geq 1)\rangle.
\eeq 
Then ${\rm Geo}(aa\inv)$ is not a rational language.
\ee

Indeed, it is straightforward to check that 
$${\rm Geo}(aa\inv) = \{ 1,a\} \cup b^* \cup \{ b^{n^2}c \mid n \geq 1 \},$$
and so it fails the conditions of the pumping lemma, satisfied by all rational languages. 

Furthermore as both $aa\inv b^{n^2}cc\inv b^{-n^2}$ and $aa\inv$ reduce to $1$ in the free group, by \cite[Lemma 1.6]{MM} the Sch\"utzenberger graphs of the inverse monoid defined by the presentation in Example \ref{irrtree} are trees.

\bt
\label{phrat}
Let ${\rm Inv} \langle A \mid r_1=s_1,\ldots, r_m=s_m \rangle$ be a finite inverse monoid presentation, and let $w \in (A \cup A\inv)^*$. If $\S(w)$ is tree-like, then ${\rm Geo}(w)$ is a rational language.
\et

\proof
Since $\S(w)$ is tree-like, there exists a constant $\delta$ satisfying condition (6) in Proposition \ref{prop_tree}.
We follow the same notation as before: put $x_0=(ww\inv)\tau$ and $K = \max\{2, |r_1|,\ldots,|r_m|,|s_1|,\ldots,|s_m|\}$.
If $S(w)$ is finite, then so is ${\rm Geo}(w)$, so it is certainly rational. Otherwise,
let $x \in V(\S(w))$ be such that 
$d(x,x_0) > \delta + K$. The idea of the proof is establishing that $D_{\delta+K}(x)$ and a finite amount of adjoint information determine $C(x_0,x)$ completely, yielding finitely many cone types. This enables us to factor the possibly infinite automaton of all geodesics to one with finitely many states.

Let $\theta_x:V(D_{\delta +K}(x)) \to \mathbb{Z}$ be defined by $\theta_x(y) = d(y,x_0)-d(x,x_0)$. It follows from the triangle inequality that $|\theta_x(y)| \leq \delta+K$. 

We write $D_{\delta +K}(x) \equiv D_{\delta +K}(x')$ if there exists an isomorphism of vertex-colored edge-labeled graphs $\p:D_{\delta +K}(x) \to D_{\delta +K}(x')$ such that $\theta_x = \theta_{x'} \circ \p$ and $\varphi(x)=x'$.
This means that, in addition to them being isomorphic, the distribution of their vertices with respect to the different distance levels from $x_0$ differs only in the constant $d(x,x_0)-d(x',x_0)$. 

Note that for any $x$, $|D_{\delta +K}(x)|$ is at most $1+2|A|+2|A|^2+\ldots+2|A|^{\delta+K}$, and there are only finitely many ways $D_{\delta +K}(x)$ could possibly be colored and a $V(D_{\delta +K}(x)) \to \mathbb Z$ function could possibly be defined, therefore there are only finitely many equivalence classes corresponding to the relation $\equiv$.

\bl
\label{phrat2}
For any vertices $x,x' \in \S(w) \setminus D_{\delta}(x_0)$, if $D_{\delta +K}(x) \equiv D_{\delta +K}(x')$, then there exists an isomorphism $\psi \colon C(x_0,x) \to C(x_0,x')$ such that $\psi(x)=x'$.
\el

Let us extend the domain of $\theta_x$ to $\cc0{D_{\delta +K}(x)}$ using the same formula. We have seen in Lemma \ref{phrat1} that $\varphi$ extends to an isomorphism $\tilde\varphi$ between $\cc0{D_{\delta +K}(x)} \cup D_{\delta +K}(x)$ and $\cc0{D_{\delta +K}(x')} \cup D_{\delta +K}(x')$. We claim that $\theta_x = \theta_{x'} \circ \tilde\p$ also holds.
Indeed, let $y \in \cc0{D_{\delta +K}(x)}$, and take a geodesic $[x_0, y]$. This intersects $D_\delta(x)$, let $v$ be a vertex in the intersection. 
Then 
\begin{eqnarray*}
\theta_x(y)&=&d(y,x_0)-d(x,x_0)=d(y,v)+d(v,x_0)-d(x,x_0)=d(y,v)+\theta_x(v)\\
&=&d(\tilde\varphi(y),\tilde\varphi(v))+\theta_{x'}(\tilde\varphi(v))=d(\tilde\varphi(y), \tilde\varphi(v))+d(\tilde\varphi(v),x_0)-d(x',x_0)\\
&\geq&d(\tilde\varphi(y),x_0)-d(x',x_0)=\theta_{x'}(\tilde\varphi(y)).
\end{eqnarray*}
The opposite inequality follows from a similar argument.

\qed

This implies that $\S(w)$ has finitely many isomorphism classes of cones, as we only have finitely many $\equiv$-classes of $(\delta+K)$-discs.

We define an equivalence relation $\rho$ on $\S(w)$ as follows. Given $x,x' \in \S(w)$, we write $x\ \rho\ x'$ if $x = x'$ or
$$x,x' \notin D_{\delta}(x_0) \mbox{ and } C(x_0,x) \stackrel{\psi}{\cong} C(x_0,x') \mbox{ with } \psi(x)=x'.$$

Let $\S_{\rm Geo}(w)$ be the subgraph of $\S(w)$ spanned by all geodesics $[x_0,x]$ where $x \in V(\S(w))$ (inverse edges are now not included). If we set $x_0$ as initial vertex and all vertices as terminal, we get an automaton $\A_{Geo}(w)$ which recognizes ${\rm Geo}(w)$. We claim that $\rho$ is a congruence of $\A_{Geo}(w)$.

Indeed, suppose $p_1 \mapright{a} q_1$ and $p_2 \mapright{a} q_2$ are edges of $\A_{Geo}(w)$ for which $p_1\ \rho\ p_2$, we need that $q_1\ \rho\ q_2$ also holds. If $p_1=p_2$, then $q_1=q_2$ and the implication is clear. Suppose that $p_1,p_2 \notin D_{\delta}(x_0)$, and suppose there exists an isomorphism $\psi: C(x_0,p_1) \to C(x_0,p_2)$ with $\psi(p_1)=p_2$. Note that $C(x_0, q_i) \subseteq C(x_0,p_i)$, and as $\psi$ respects the labels, $\psi(q_1)=q_2$, and $\psi(C(x_0, q_1))=C(x_0, q_2)$. Thus we have $q_1\ \rho\ q_2$ as needed.

It follows that we can factor $\A_{Geo}(w)$ by the automaton congruence $\rho$ and get a quotient automaton $\A_{Geo}(w)/\rho$ recognizing the same language. Since $\rho$ is a finite index congruence by Lemma \ref{phrat2} and the remark above it, the automaton is finite, hence ${\rm Geo}(w)$ is rational as claimed.
\qed

The theorem implies that the inverse monoid in Example \ref{irrtree} cannot be finitely presented, yielding the following corollary:

\bc
\label{nonfp}
There exists a finitely generated tree-like inverse monoid which is not finitely presentable.
\ec

\subsection{The word problem}

We now go on to prove that finitely presented tree-like inverse monoids have a uniformly solvable word problem, moreover, the language of their Sch\"utzenberger automata is not only decidable, but context-free.

Let $ {\rm Inv} \langle A \mid r_1=s_1,\ldots, r_m=s_m \rangle$ be a finite inverse monoid presentation fixed for the rest of the section, and put $K = \max\{2,|r_1|,\ldots,|r_m|,|s_1|,\ldots,|s_m|\}$. Let us fix a word $w \in (A \cup A^{-1})^\ast$. 
{We call a graph $\Gamma$ \emph{$P$-complete} if no $P$-expansions can be performed on it. Given any graph $\Gamma$, if $\Gamma'$ is a subgraph of $\Gamma$ such that for any relation $r=s$, and any path $p_s$ in $\Gamma'$ labeled by $s$, there exists a parallel path labeled by $t$ in $\Gamma$, we say that $\Gamma'$ is \emph{relatively $P$-complete} in $\Gamma$.}

Before giving the precise proofs, we outline the very rough idea used throughout the section. 
{Due to the tree-like geometry of $\S(w)$, during the usual process of approximating $\S(w)$, eventually we get to a point where in our finite deterministic approximating graph $S$, there exist some subgraphs $Y_1, \ldots, Y_n$ of a universally bounded diameter such that $\bigcap\limits_{i=1}^{n} \x0{Y_i}$ is relatively $P$-complete in $S$.} Recall that by Lemma \ref{explem}, the set $\Gamma^c(x_0, Y_i)$ of $\S(w)$ only depends on $Y_i^{+K}$. Suppose that an isomorphic copy $X_i^{+K}$ of $Y_i^{+K}$ exists in the relatively $P$-complete part. Then $\S(w)$ can be obtained from the approximate graph by iteratively gluing copies of $\Gamma^c(x_0, X_i)$ to $Y_i^{+K}$ in the way the isomorphism dictates. We will call approximate graphs with the above property sapling, and the formal definition is the following.

\begin{figure}
\includegraphics[width=0.7\linewidth]{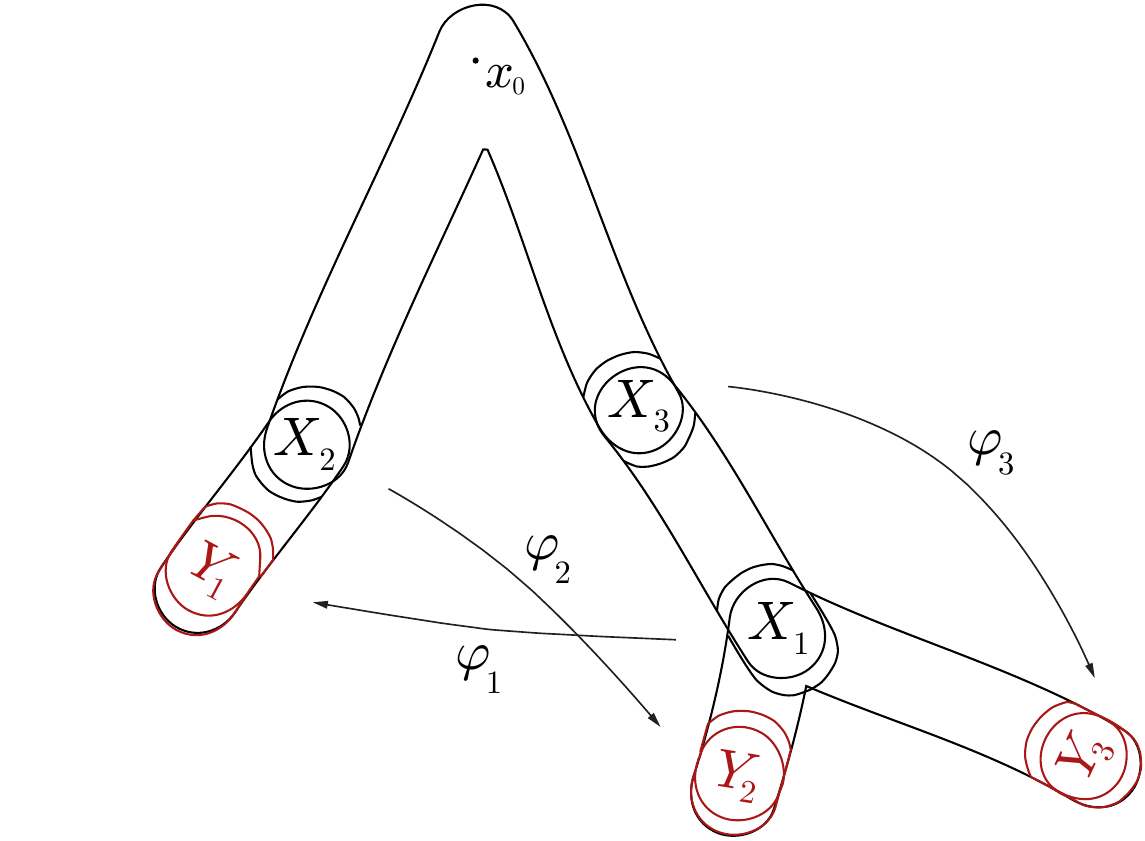}
\caption{A sapling $S$, with the sets $X_i$, $Y_i$, $X_i^{+K}$ and $Y_i^{+K}$ circled.}
\end{figure}

\begin{Def}
\label{saplingdef}
Let $S$ be a finite deterministic approximate graph of $\S(w)$. Then $S$ contains a path starting at $x_0$, labeled by $w$ -- denote this by $M$. We call $S$ a \emph{sapling} if there exist induced connected subgraphs $Y_1,\ldots,Y_n$ of $S \setminus M$ such that
\begin{enumerate}
\item $\bigcap\limits_{i=1}^{n} \x0{Y_i}$ is relatively $P$-complete in $S$
\item for all $1\leq i \leq n$, $\cc0{Y_i}\subseteq Y_i^{+K}$,
\item for any $1 \leq i < j \leq n$, $d(\cc0{Y_i} \cup Y_i,\cc0{Y_j} \cup Y_j)\geq 2$,
\item for all $1\leq i \leq n$, there exists an induced subgraph $X_i$ of $S \setminus M$  such that $X_i^{+K} \subseteq  \bigcap\limits_{j=1}^{n}\x0{Y_j}$, and $X_i^{+K}\cong_c Y_i^{+K}$ 
\item there exists some $k \in \mathbb N$ such that
$X_i \cup \cc0{X_i}$ is isomorphic to a subgraph of $\exp_k(X_i^{+K})$ via an isomorphism that fixes $X_i$.
\end{enumerate}
\end{Def}

For the last condition, we can identify $\cc0{X_i}$ with its image under the isomorphism and write $\cc0{X_i} \subseteq \exp_k(X_i^{+K})$ for short.

Let $S$ be a sapling, and consider the sets $X_1, \ldots, X_n$, $Y_1, \ldots, Y_n$ as given by the definition. Denote the color preserving isomorphisms $X_i^{+K} \to Y_i^{+K}$ by $\varphi_i$. We define a new graph $\tilde{S}$ by ``gluing'' new copies of $\cc0{X_i}$ to $Y_i^{+K}$ as prescribed by $\varphi_i$. More precisely, for each $i$, let $\psi_i(\cc0{X_i})$ be an isomorphic copy of $\cc0{X_i}$ under the isomorphism $\psi_i$, disjoint from $S$ and from each other. Consider the graph $S \cup \bigsqcup_i \psi_i(\cc0{X_i})$, and let $\sim$ be an equivalence relation on its sets of edges and vertices generated by $\varphi_i(t) \sim \psi_i(t)$ whenever $t$ is a vertex or edge of $X_i^{+K} \cap \cc0{X_i}$. Define 
$$\tilde{S}=(S \cup \bigsqcup_i \mathop{}\psi_i(\cc0{X_i}))/\sim.$$
Note that as $\sim$ restricted to $S$ is the identity relation, $S \subseteq \tilde{S}$, and similarly, $\psi_i(\cc0{X_i}) \subseteq \tilde S$.

\begin{figure}
\includegraphics[width=\linewidth]{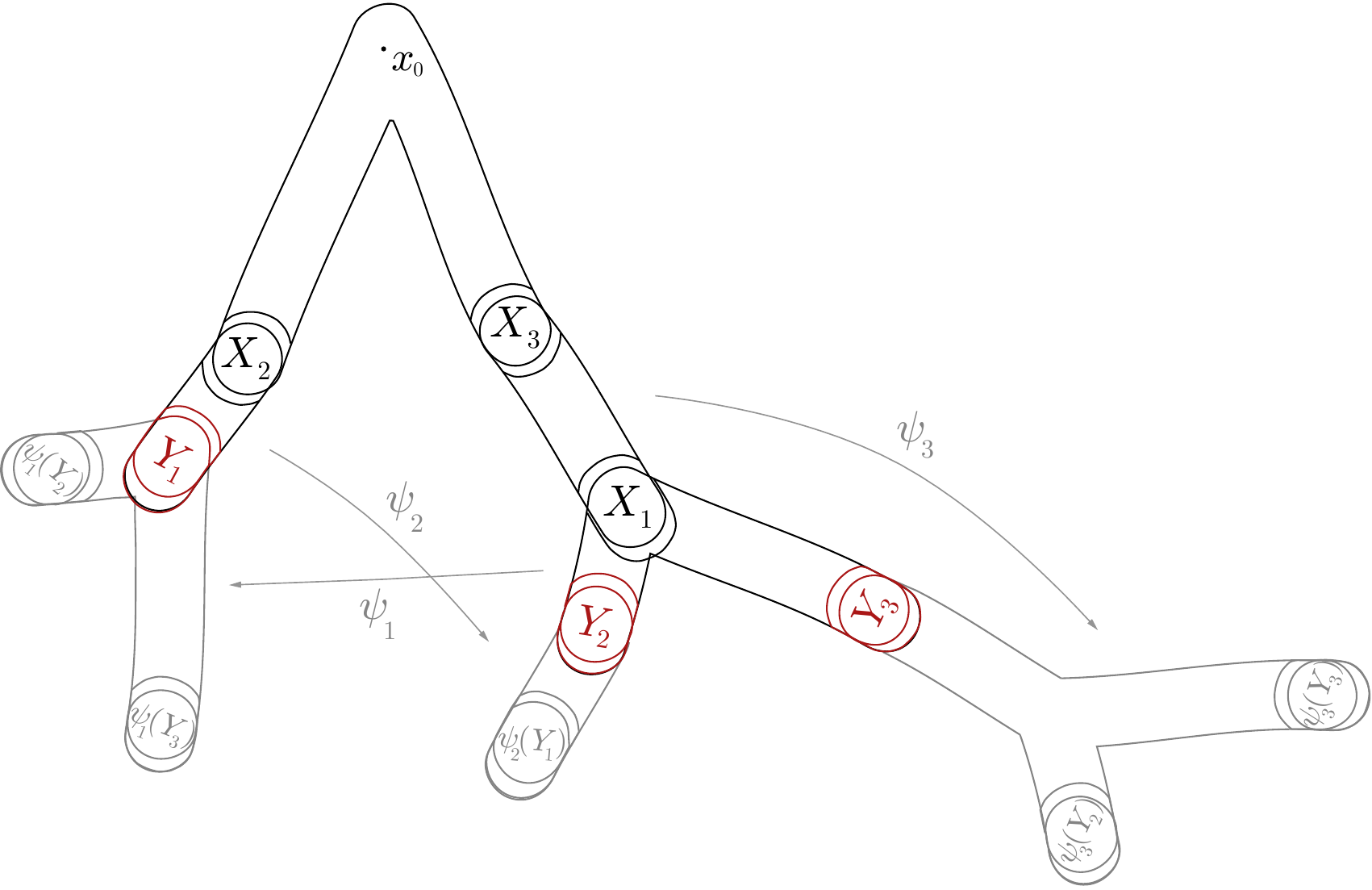}
\caption{A new sapling $\tilde{S}$.}
\end{figure}

In the sequel, we reserve the notations $\Gamma(x_0,X)$ $(X \subseteq V(S))$ to denote the respective subgraphs defined in $S$. For the similarly defined subgraphs of $\tilde{S}$, we will use the notation $\tilde{\Gamma}(x_0,X)$.

Notice that as $\varphi$ was color-preserving, $\varphi_i \cup \psi_i$ is a color-preserving isomorphism from $X_i^{+K} \cup \Gamma^c(x_0,X_i)$ to $Y_i^{+K} \cup \tilde{\Gamma}^c(x_0,Y_i)$, and $(\varphi_i \cup \psi_i)(X_i^{+K})=Y_i^{+K}$. Furthermore $(\varphi_i \cup \psi_i)(\Gamma^c(x_0,X_i))=\tilde{\Gamma}^c(x_0,Y_i)$: indeed, since $\Gamma^c(x_0,X_i)$ is the component of the red vertices in $(X_i^{+K} \cup \Gamma^c(x_0,X_i)) \setminus X_i$, we have that  $(\varphi_i \cup \psi_i)(\Gamma^c(x_0,X_i))$ is the component of the red vertices of $Y_i^{+K}$ in $(Y_i^{+K} \cup \tilde{\Gamma}^c(x_0,Y_i)) \setminus Y_i$, which is by definition $\tilde{\Gamma}^c(x_0,Y_i)$. We also have $(\varphi_i \cup \psi_i)(\Gamma^c(x_0, Y_j))=\psi_i(\Gamma^c(x_0, Y_j))=\tilde{\Gamma}^c(x_0,\psi_i(Y_j))$.

\begin{Lemma}
\label{saplem}
If $S$ is a sapling, then so is $\tilde{S}$.
\end{Lemma}

\proof
The first thing we need to check is that $\tilde{S}$ is deterministic. Seeking a contradiction, suppose it is not. As $S$ and $\bigsqcup_i \psi_i(\Gamma^c(x_0,X_i))$ are both deterministic, the only possible way for $\tilde{S}$ to not be deterministic is if $\sim$ identifies the initial vertices of two edges of the same label, but not the edges themselves. {Since the sets $\cc0{Y_i}$ are pairwise disjoint, so are the images of $\psi_i(\Gamma^c(x_0,X_i))$ in $\tilde S$, therefore such an identification can only involve an edge of $S$ and an edge of $\psi_i(\Gamma^c(x_0,X_i))$ for some $i$.} Suppose this is the case, let $e \in S$ and $f \in \psi_i(\Gamma^c(x_0,X_i))$ be such edges. 
Then $\alpha(e) \in Y_i^{+K} \cap \Gamma^c(x_0,Y_i)$, and $\alpha(f) \in \psi_i(X_i^{+K} \cap \Gamma^c(x_0,X_i))$ with $\psi_i^{-1}(\alpha(f))=\varphi_i^{-1}(\alpha(e))$. Note that $e \in Y_i^{+K}$. Indeed, if $e \notin Y_i^{+K}\cap \Gamma^c(x_0,Y_i)=\Gamma^c(x_0,Y_i)$, then $\omega(e) \in Y_i$, hence $e \in Y_i^{+K}$. Therefore, $\varphi_i^{-1}(e)$ is an edge in $X_i^{+K}$ starting at $\psi_i^{-1}(\alpha(f))$ with the same label as $f$, hence $\varphi_i^{-1}(e)=\psi_i^{-1}(f)$, so $e$ and $\psi_i(f)$ should be identified by $\sim$. 

To see that $\tilde{S}$ is an approximate graph of $S(w)$, observe that because of the isomorphism $\varphi_i \cup \psi_i$ from $X_i^{+K} \cup \Gamma^c(x_0,X_i)$ to $Y_i^{+K} \cup \tilde{\Gamma}^c(x_0,Y_i)$, since $\Gamma^c(x_0,X_i) \subseteq \exp_k(X_i^{+K})$, we have $\tilde{\Gamma}^c(x_0,Y_i) \subseteq \exp_k(Y_i^{+K})$, that is, all new edges of $\tilde S$ are obtainable from $S$ using $P$-expansions and foldings.

Next we show that the sets $\tilde Y_{ij}=\psi_i(Y_j)$, where $1\leq i,j\leq n$ such that $Y_j \subseteq \cc0{X_i}$ satisfy the properties listed in Definition \ref{saplingdef}.

\textit{1.} Suppose $s=t$ is a relation in $R$, and suppose $s$ labels a path $p_s$ in $\bigcap _{i,j} \tilde{\Gamma}(x_0,\psi_i(Y_j))$. 
If $p_s$ is contained in $\bigcap_i \tilde{\Gamma}(x_0,Y_i)=\bigcap_i \Gamma(x_0,Y_i)$, then by assumption there is a path parallel to $p_s$ labeled by $t$ in $S \subseteq \tilde S$, and we are done. 
Suppose this is not the case. Then $p_s$ must contain a vertex in $Y_k\cup \tilde{\Gamma}^c(x_0,Y_k)$ for some $1 \leq k \leq n$, hence $p_s$ is contained in $(Y_k^{+K} \cup \tilde{\Gamma}^c(x_0,Y_k)) \cap \bigcap _{i,j} \tilde{\Gamma}(x_0,\psi_i(Y_j))$. This graph is isomorphic to $(X_k^{+K} \cup {\Gamma}^c(x_0,X_k)) \cap \bigcap _{i} {\Gamma}(x_0,Y_i)$ by $(\varphi_k \cup \psi_k)^{-1}$. 
Let $p_t$ be the path in $S$ labeled by $t$, parallel to the isomorphic copy of $p_s$ in $\bigcap _{i} {\Gamma}(x_0,Y_i) \cap (X_k^{+K} \cup {\Gamma}^c(x_0,X_k))$. 
Consider the closed path $p_s p_t^{-1}$, this is of length at most $2K$ and contains a vertex in $X_k \cup {\Gamma}^c(x_0,X_k)$, therefore it must be contained in $X_k^{+K} \cup {\Gamma}^c(x_0,X_k)$, hence $(\varphi_k \cup \psi_k)(p_t)$ is the path labeled by $t$ parallel to $p_s$.

\textit{2.} For all $1\leq i,j \leq n$, $\tilde{\Gamma}^c(x_0,\tilde Y_{ij})\subseteq \tilde Y_{ij}^{+K}$: 
since $S$ is a sapling, we have $\psi_i(\Gamma^c(x_0,Y_j))\subseteq \psi_i(Y_j^{+K})$, hence
$$\tilde{\Gamma}^c(x_0,\psi_i(Y_j))=\psi_i(\Gamma^c(x_0,Y_j))\subseteq \psi_i(Y_j^{+K})=\psi_i(Y_j)^{+K}.$$

\textit{3.} For any $1 \leq i_1,i_2, j_1, j_2 \leq n$, if $(i_1,j_1)\neq (i_2, j_2)$, then
$d(\tilde{\Gamma}^c(x_0, \tilde{Y}_{i_1j_1}) \cup \tilde{Y}_{i_1j_1},\tilde{\Gamma}^c(x_0, \tilde{Y}_{i_2j_2}) \cup \tilde{Y}_{i_2j_2})\geq 2$:

Note that for any $i,j$, we have $\tilde{\Gamma}^c(x_0,\psi_i(Y_j)) \cup \psi_i(Y_j)=\psi_i(\Gamma^c(x_0,Y_j) \cup Y_j)\subseteq \cc0{Y_i}$, hence 
if $i_1 \neq i_2$, then 
$$d(\psi_{i_1}(\Gamma^c(x_0,Y_{j_1}) \cup Y_{j_1}),\psi_{i_2}(\Gamma^c(x_0,Y_{j_2}) \cup Y_{j_2}))\geq d(\cc0{Y_{i_1}}, \cc0{Y_{i_2}})\geq 2$$
by assumption.

For the case of $i_1=i_2:=i$, note that
$d(\psi_{i}(\Gamma^c(x_0,Y_{j_1}) \cup Y_{j_1}),\psi_{i}(\Gamma^c(x_0,Y_{j_2}) \cup Y_{j_2}))\geq 2$ is equivalent to saying $(\psi_{i}(\Gamma^c(x_0,Y_{j_1}) \cup Y_{j_1}))^{+1} \cap (\psi_{i}(\Gamma^c(x_0,Y_{j_2}) \cup Y_{j_2}))^{+1}=\emptyset$,
but the left hand side is equal to 
$(\psi_{i}(\Gamma^c(x_0,Y_{j_1}) \cup Y_{j_1})^{+1} \cap (\Gamma^c(x_0,Y_{j_2}) \cup Y_{j_2}))^{+1}=\psi_{i}(\emptyset)=\emptyset$ by assumption.


\textit{4.} We show that the sets $Y_j$ with the new indexing $\tilde X_{ij}=Y_j$, where $1\leq i,j\leq n$ are such that $Y_j \in \cc0{X_i}$, satisfy the conditions. Clearly $\tilde X_{ij} \cap M=\emptyset$, as $\cc0{X_i} \cap M =\emptyset$.

For all $i,j$, $\tilde X_{ij}^{+K} \cong_c \tilde Y_{ij}^{+K}$: by $\psi_i(Y_j^{+K})=\psi_i(Y_j)^{+K}$, the edge-labeled graph $\tilde X_{ij}^{+K}=Y_j^{+K}$ is isomorphic to $\tilde Y_{ij}^{+K}=\psi_i(Y_j)^{+K}$ by $\psi_i$. Furthermore since $\tilde{\Gamma}^c(x_0,\psi_i(Y_j))=\psi_i(\Gamma^c(x_0,Y_j))$, they are also isomorphic as vertex-colored, edge-labeled graphs.

$\tilde X_{ij}^{+K} \subseteq \bigcap_{k,l} \tilde\Gamma(x_0, \tilde Y_{kl})$: note that  $\psi_k(Y_l) \cap S = \emptyset$ for all $k,l$, hence $S \subseteq \bigcap_{k,l} \tilde\Gamma(x_0, \psi_k(Y_l))$, in particular, $Y_j^{+K} \subseteq \bigcap_{k,l} \tilde\Gamma(x_0, \psi_k(Y_l))$.

{\textit{5.}
$\tilde{\Gamma}^c(x_0,\tilde{X}_{ij}) \subseteq \exp_k(\tilde X_{ij}^{+K})$:
$$\tilde{\Gamma}^{c}(x_0,Y_j)=\psi_j(\Gamma^c(x_0, X_j))= \psi_j(\exp_k(X_j^{+K})\cap \Gamma^{c}(x_0,X_j))\subseteq \exp_k(\psi_j(X_j^{+K}))=\exp_k(Y_j^{+K}).$$}
\qed

\bt
\label{sapunion}
Let $\invp\langle A \mid r_1=s_1, \ldots, r_m=s_m\rangle$ be a finite presentation, and let $S$ be a sapling for the word $w \in (A \cup A)^\ast$. Suppose $K = \max\{ |r_1|,\ldots,|r_m|,|s_1|,\ldots,|s_m|\} > 1$. We define a growing sequence of graphs by $S_0=S$, and $S_{i+1}=\tilde S_{i}$ for $i=0,1,2,\ldots$. Then 
$$\bigcup_{i=0}^\infty S_i=\S(w).$$
\et

\proof
It is clear that $\S'(w)=\bigcup\limits_{i=1}^{\infty} S_i$ is an approximate graph of $\S(w)$, as $S_i$ is an approximate subgraph for all $i \in \mathbb N$. 
Furthermore, $\S'(w)$ is deterministic, and 
the construction ensures that $S_i$ is relatively $P$-complete in $S_{i+1}$ for all $i \in \mathbb N$. For any relation $s=t$, and any path $p_s$ in $\S'(w)$ labeled by $s$, the path $p_s$ is contained in a subgraph of the form $S_j$ for some $j$, hence $S_{j+1}\subseteq S'(w)$ a parallel path $p_t$ labeled by $t$. This proves that $\S'(w)$ is also $P$-complete.
\qed

\brm
\label{sapremark}
Note that this implies that a sapling is a finite induced subgraph of $\S(w)$. 
\erm

\bp
\label{sapexist}
If $\invp\langle A \mid R\rangle$ is a finite presentation, and $w \in (A \cup A^{-1})^\ast$ is a word for which $S(w)$ is infinite and tree-like, then there exists a sapling for $w$. 
\ep

\proof
Consider a strong tree decomposition $\P$ of the graph $\S(w)$, and let $T=\S(w)/\P$. Denote the map $V(\S(w)) \to V(T)$, $x \mapsto x/\P$ by $\pi$. We will consider $T$ a rooted tree with root $\pi(x_0)$.
Let $m$ be such that for all sets $X \in \P$, $\operatorname{diam}(X) \leq m$.
There are finitely many $(A \cup A^{-1})$-labeled, red-blue vertex colored digraphs with diameter at most $m+K$. Hence there exists $r \in \mathbb Z^+$ such that for all the labeled digraphs of the form $X^{+K}$, where $X \in P$, an isomorphic copy of the form $X'^{+K}$, where $X' \in P$,  is contained in $D_r(x_0) \subseteq \S(w)$. Assume furthermore that $r$ is greater then the length of $w$. 
Let $y_1, \ldots, y_n$ be the set of vertices of $T\setminus\pi(D_r(x_0))$ which are connected to the leaves of $\pi(D_r(x_0))$ --- there are finitely many such vertices, as $T$ is locally finite. Let $Y_i=\pi^{-1}(y_i)$, and let $D=\pi^{-1}(\pi(D_r(x_0)))$. Note that $D \supseteq D_r(x_0)$. Put $S=D \cup \bigcup\limits_{i=1}^{n}{Y_i}^{+K}$. We claim that $S$ is a sapling with sets $Y_1, \ldots, Y_n$.

It is certainly true that $S$ is a finite connected approximate graph of $\S(w)$.
We reserve $\x0{X}$ to denote subsets of $\S(w)$, the similarly defined subsets of $T$ and $S$ will be denoted by  $\Gamma_T(\pi(x_0),X)$ and $\Gamma_S(x_0, X)$.
Note that for any set $X \in \P\setminus{\pi^{-1}\pi(x_0)}$, $\x0{X}=\pi^{-1}(\Gamma_T(\pi(x_0),\pi(X)))$, and similarly, $\cc0{X}=\pi^{-1}(\Gamma_T^c(\pi(x_0),\pi(X)))$. 
The defining conditions of saplings are satisfied as follows:

1. Note that $\bigcap\limits_{i=1}^n \Gamma_S(x_0, Y_i)=D$. For any path labeled by a relator in $D$, there exists a parallel path labeled by its pair in $\S(w)$, and this is contained in $D^{+K}\subseteq S$ indeed.

2. For any $i$, $\Gamma^c_S(x_0, Y_i)=\Gamma^c(x_0, Y_i) \cap S \subseteq Y_i^{+K}$ by the definition of $S$.

3. In $T$, $d(\Gamma_T^c(\pi(x_0),y_i)\cup y_i, \Gamma_T^c(\pi(x_0),y_j)\cup y_j)\geq 2$ $(1 \leq i<j \leq n)$ by the choice of the $y_i$s, hence $d(\cc0{Y_i} \cup Y_i,\cc0{Y_j} \cup Y_j)\geq 2$
also holds in $\S(w)$, and in $S$.

4. There are certainly sets of vertices $X_i \in P$ for any $i=1, \ldots, n$ with $X_i^{+K}\cong_c Y_i^{+K}$ and $X_i^{+K} \subseteq D$ by the definition of $D$. 

5. By Lemma \ref{explem}, $\cc0{X_i} \subseteq \exp(X_i^{+K})$ holds in $\S(w)$, and therefore $\Gamma^c_S(x_0, X_i)=\Gamma^c(x_0, X_i) \cap S \subseteq \exp(X_i^{+K})$. Since $\Gamma^c_S(x_0, X_i)$ is finite, there exists some $k \in \mathbb N$ such that $\Gamma^c_S(x_0, X_i) \subseteq \exp_k(X_i^{+K})$ indeed.
\qed

\begin{figure}
\includegraphics[width=\linewidth]{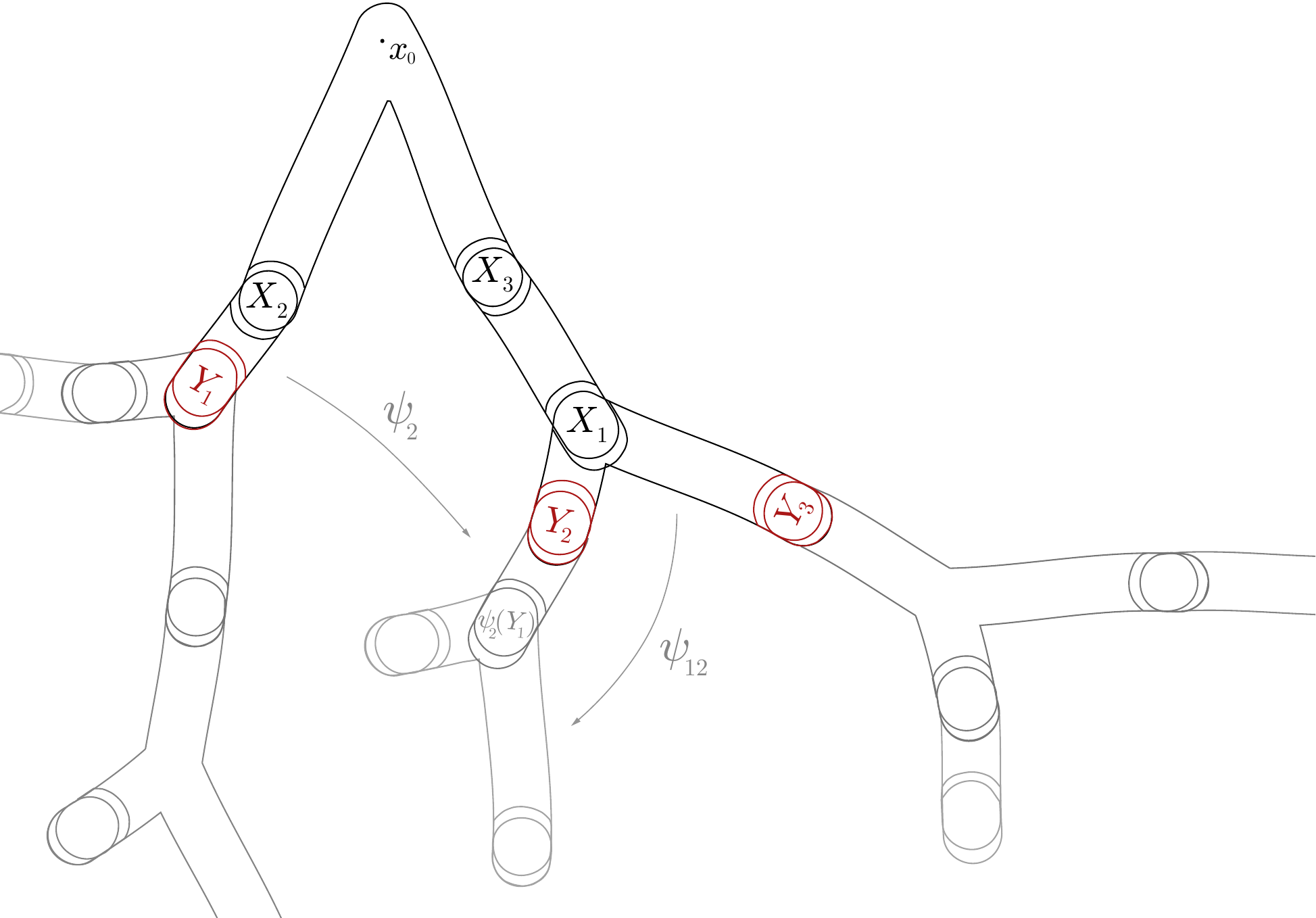}
\caption{\label{figsapmap} The sapling $S_2$}
\end{figure}

We now proceed to one of the main theorems of the section.

\bt
\label{cfthm}
For an inverse monoid given by a finite presentation $M=\invp\langle A \mid R\rangle$ and a word $w \in (A \cup A^{-1})^\ast$, if the Sch\"utzenberger graph $\S(w)$ is tree-like, then the language $L(\A(w))=\{u \in (A\cup A^{-1})^\ast: u  \geq w\}$ is context-free.
\et

\proof
If $\A(w)$ happens to be finite, then $L(\A(w))$ is regular, therefore context-free. 
Otherwise, we define a pushdown automaton $\P(w)$ that recognizes the same language as the automaton $\A(w)$.

Let $S$ be a sapling for the word $w$ with sets $Y_1, \ldots, Y_n$ --- by Proposition \ref{sapexist}, this exists. Denote the set $\{1, \ldots, n\}$ by $[n]$. In keeping with the previous notation, let $S_0=S$, and $S_i=\tilde{S}_{i-1}$ for $i \geq 1$. Denote the subgraphs of $S_i$ of the form $\x0{X}$ and  $\cc0{X}$ by $\Gamma_i(x_0,X)$ and  $\Gamma^{c}_i(x_0,X)$ respectively.

As before, for any $i \in [n]$, we have 
isomorphisms from $\Gamma_0^{c}(x_0, X_i)$ to $\Gamma_1^{c}(x_0, Y_i)$, which we denote by $\psi_i$, and $S_1=(S_0 \cup \bigsqcup_i \psi_i(\Gamma_0^{c}(x_0, X_i))/\sim$, where $\sim$ is the equivalence generated by pairs of the form $\varphi_i(t) \sim \psi_i(t)$ for any $t \in X_i^{+K} \cap \Gamma_0^c(x_0,X_i)$. Furthermore, $S_1$ is a sapling with respect to the sets $\psi_i(Y_j)$, where $Y_j \subseteq \Gamma_0^{c}(x_0, X_i)$. 

Iterating the construction to obtain $S_2=\tilde{S_1}$, for any $i \in [n]$ such that $Y_{i_2} \subseteq \Gamma_0^{c}(x_0, X_{i_1})$, we have isomorphisms from $\Gamma_1^{c}(x_0, Y_{i_2})=\psi_{i_2}(\Gamma_0^{c}(x_0, X_{i_2}))$ to $\Gamma_2^c(x_0,\psi_{i_1}(Y_{i_2}))$. Denote the induced isomorphism from $\Gamma_0^{c}(x_0, X_{i_2})$ to $\Gamma_2^c(x_0,\psi_{i_1}(Y_{i_2}))$ by $\psi_{i_2i_1}$ (see Figure \ref{figsapmap} for illustration). Note that then $S_2=(S_1 \cup \bigsqcup_{i_1,{i_2}} \psi_{i_2i_1}(\Gamma_0^{c}(x_0, X_{i_2}))/\sim$, where $\sim$ is the equivalence generated by pairs of the form  $\psi_{i_1}(\varphi_{i_2})(t) \sim \psi_{i_2i_1}(t)$ for any $t \in X_{i_2}^{+K} \cap \Gamma_0^c(x_0,X_{i_2})$. Furthermore, $S_2$ is a sapling with respect to $\psi_{i_2i_1}(Y_{i_3})$ where $Y_{i_2} \subseteq \Gamma_0^{c}(x_0, X_{i_1})$, and $Y_{i_3} \subseteq \Gamma_0^{c}(x_0, X_{i_2})$.

Suppose that, up to some positive integer $k$, we have already defined all the isomorphisms of the form $\psi_{i_k \ldots i_1}$ from $\Gamma_0^{c}(x_0, X_{i_k})$ to $\Gamma_{k}^c(x_0,\psi_{i_{k-1} \ldots i_1}(Y_{i_k}))$, where $i_j \in [n]$ and $Y_{i_j} \subseteq \Gamma_0^{c}(x_0, X_{i_{j-1}})$ for $2 \leq j \leq k$, and suppose $S_{k}$ is a sapling with sets $\psi_{i_k \ldots i_1}(Y_{i_{k+1}})$. 
Then, by the construction of $\tilde S_k= S_{k+1}$, we have isomorphisms from $\Gamma_{k}^c(x_0,\psi_{i_{k-1} \ldots i_1}(Y_{i_k}))=\psi_{i_k \ldots i_1}(\Gamma_0^{c}(x_0, X_{i_k}))$ to $\Gamma^c_{k+1}(x_0,\psi_{i_k \ldots i_1}(Y_{k+1}))$. 
Denote the isomorphism from $\Gamma_0^{c}(x_0, X_{i_{k+1}})$ to $\Gamma^c_{k+1}(x_0,\psi_{i_k \ldots i_1}(Y_{k+1}))$ by $\psi_{i_{k+1} \ldots i_1}$. 
Again, note that $S_{k+1}=(S_k \cup \bigsqcup_{i_1, \ldots, i_{k+1}} \mathop{}\psi_{i_{k+1} \ldots i_1}(\Gamma_0^{c}(x_0, X_{i_{k+1}})))/\sim$, where $\sim$ is the equivalence generated by pairs of the form  $\psi_{i_k \ldots i_1}(\varphi_{i_{k+1}})(t) \sim \psi_{i_{k+1} \ldots i_1}(t)$ for any $t \in X_{i_{k+1}}^{+K} \cap \Gamma_0^c(x_0,X_{i_{k+1}})$. Furthermore, $S_{k+1}$ is a sapling with respect to $\psi_{i_{k+1} \ldots i_1}(Y_{k+1})$, where $Y_{i_j} \subseteq \Gamma_0^{c}(x_0, X_{i_{j-1}})$ for $2 \leq j \leq k+1$.

By this iterative construction of $\S(w)$, we have 

$$\S(w)=S_0 \cup \bigcup_{\substack{k \in \mathbb N_0 \\ Y_{i_j} \subseteq \Gamma_0^{c}(x_0, X_{i_{j-1}})}}\psi_{i_k \ldots i_1}(\Gamma_0^{c}(x_0, X_{i_k})).$$

In the pushdown automaton $\P(w)$, we only keep $S_0$ and just one extra isomorphic copy of each $\Gamma_0^{c}(x_0, X_i)$, and we keep track of which isometric copy we would be in $\A(w)$ by keeping the corresponding string $i_1 \ldots i_k$ in the stack. The precise definition is as follows.

Let $\psi_i(\Gamma_0^{c}(x_0, X_i))$ be disjoint isomorphic copies of $\Gamma_0^{c}(x_0, X_i)$ by the isomorphism $\psi_i$ for $i \in [n]$, and let $f$ be a single vertex disjoint from both of these. Consider the graph $P=S_0 \cup \bigsqcup_i \psi_i(\Gamma_0^{c}(x_0, X_i))$. The set of states for $\P(w)$ is $V(P)$, the input alphabet is of course $X \cup X^{-1}$, the stack alphabet is $[n]=\{1,\ldots, n\}$, $Z$ is the initial stack symbol, the inital state is $x_0$, the terminal state is $w \in S_0$. The transitions are as follows. Suppose $u \xrightarrow{a}  v$ is a transition of $\A(w)$ with $u, v \in V(S)$. Then we define the following transitions in $\P(w)$:
\begin{itemize}
\item let $u \xrightarrow{a,\ k \to k} v$ for any $k \in [n]$,
\item furthermore if $u,v \in V(\cc0{X_i})$ let $\psi_i(u) \xrightarrow{a,\ k \to k}  \psi_i(v)$ for any $i,k \in [n]$.
\end{itemize}
In other words, we keep the transitions inherited from $\A(w)$ as they are.
In addition, we have the following transitions whenever $u \in V(X_i^{+K} \cap \cc0{X_i})$:
\begin{itemize}
\item $\varphi_i(u) \xrightarrow{\epsilon,\ k \to ik} \psi_i(u)$ for any $i,k \in [n]$,
\item $\psi_i(u) \xrightarrow{\epsilon,\ i \to \epsilon} \varphi_i(u)$ for any $i,k \in [n]$, 
\item $\psi_i(\varphi_j(u)) \xrightarrow{\epsilon,\ k \to jk} \psi_j(u)$ for any $i,j,k \in [n]$ with $Y_j \subseteq \cc0{X_i}$,
\item $\psi_j(u) \xrightarrow{\epsilon,\ j \to \epsilon} \psi_i(\varphi_j(u))$ for any $i,j,k \in [n]$ with $Y_j \subseteq \cc0{X_i}$.
\end{itemize}

We claim that $L(\P(w))=L(\A(w))$. 

For the proof, consider the following non-deterministic automaton $\A'(w)$.
The set of states is $$V(S_0\ \sqcup \ {\bigsqcup_{\substack{k \in \mathbb N_0 \\ Y_{i_j} \subseteq \Gamma_0^{c}(x_0, X_{i_{j-1}})}}}\phi_{i_k \ldots i_1}(\Gamma_0^{c}(x_0, X_{i_k}))),$$
where $\phi_{i_k \ldots i_1}(\Gamma_0^{c}(x_0, X_{i_k}))$ is an isomorphic copy of $\Gamma_0^{c}(x_0, X_{i_k})$ by $\phi_{i_k \ldots i_1}$, all disjoint from each other and $S_0$ --- essentially these correspond to the sets $\psi_{i_k \ldots i_1}(\Gamma_0^{c}(x_0, X_{i_k}))$ in $\A(w)$. The initial state is, again, $x_0$, the terminal state is $w$.
The set of transitions is the following. If $u,v \in V(S_0)$ and $u \xrightarrow{a}  v$ in $\A(w)$, then let $u \xrightarrow{a}  v$ in $\A'(w)$, and if in addition $u,v \in V(\Gamma_0^{c}(x_0, X_{i_k}))$, let $\phi_{i_{k} \ldots i_1}(u) \xrightarrow{a} \phi_{i_{k} \ldots i_1}(v)$. Furthermore, if for some $k$ and $i_1, \ldots, i_k$, $\psi_{i_k \ldots i_1}(u)=\psi_{i_{k-1} \ldots i_1}(\varphi_k(u))$, put $\phi_{i_k \ldots i_1}(u) \xrightarrow{\epsilon} \phi_{i_{k-1} \ldots i_1}(\varphi_k(u))$ and $\phi_{i_{k-1} \ldots i_1}(\varphi_k(u)) \xrightarrow{\epsilon} \phi_{i_k \ldots i_1}(u)$.
Clearly, $L(\A'(w))=L(A(w))$, as $A(w)$ is essentially obtained by $\A'(w)$ by identifying the vertices between which there are $\epsilon$-transitions.

Recall that the configurations of $\P(w)$ are pairs $(q,\alpha Z)$, where $q$ is a state of $\P(w)$, and $\alpha \in [n]^\ast$ is the string in the stack written from top down up until the bottom symbol. Consider the following one-to-one correspondence between the set of all configurations of $\P(w)$ and the set of all states of $\A'(w)$:
\begin{eqnarray*}
(u, Z) &\leftrightarrow& u, \hbox{ where } u \in S_0\\
(\psi_{i_k}(u), {i_k \ldots i_1}Z) &\leftrightarrow& \phi_{i_k \ldots i_1}(u), \hbox{ where } u \in S_0, {i_k \ldots i_1} \in [n], Y_{i_j} \subseteq \Gamma_0^{c}(x_0, X_{i_{j-1}}).
\end{eqnarray*}
It is straightforward to check that this correspondence respects transitions in the sense that reading a new letter from the input will change corresponding configurations to corresponding configurations. It is also true that the initial vertices and respectively, the terminal vertices correspond in the two automata. Therefore $L(\P(w))=L(\A'(w))=L(\A(w))$, which proves the statement.
\qed

\bt
\label{esaptreelike}
If $\invp\langle A \mid R\rangle$ is a finite presentation and $w \in (A \cup A^{-1})^\ast$ is such that there exists a sapling for $S(w)$, then $S(w)$ is tree-like.
\et

\proof
By Proposition \ref{prop_tree} it will suffice to prove that $S(w)$ has a strong tree decomposition of finite width. Suppose $S$ satisfies the definition of sapling with the subgraphs $X_1, \ldots, X_n, Y_1, \ldots, Y_n$. We follow the notation introduced in the previous proof.  

For $k \geq 2$ and $j_1, \ldots, j_k \in [n]$ such that $Y_{j_l} \subseteq \Gamma_0^c(x_0, X_{j_{l-1}})$,
let $Y_{j_k\ldots j_1}=\psi_{j_{k-1}\ldots j_1}(Y_{j_k})$.

Observe that 
$$V(S(w))=V\Big(\bigcap_{j_1}(\x0{Y_{j_1}})\cup \bigcup\limits_{j_1}Y_{j_1}\Big)\ \sqcup\ {\bigsqcup\limits_{j_1}}V(\cc0{Y_{j_1}}),$$
furthermore, for any $j_1 \in [n]$,
\begin{eqnarray*}
V(\cc0{Y_{j_1}})&=&V(\cc0{Y_{j_1}})\cap \bigg(V\Big(\bigcap_{j_2}(\x0{Y_{j_2j_1}})\cup \bigcup\limits_{j_2}Y_{j_2j_1}\Big)\ \sqcup\ {\bigsqcup\limits_{j_2}}V(\cc0{Y_{j_2j_1}})\bigg)\\
&=&V\Big(\bigcap_{j_2}(\cc0{Y_{j_1}}\cap \x0{Y_{j_2j_1}})\cup \bigcup_{j_2} Y_{j_2j_1}\Big)\ \sqcup\ {\bigsqcup_{j_2}} V(\cc0{Y_{j_2j_1}})
\end{eqnarray*}}
(In order to keep the formulae from becoming overly convoluted, we only specify the variable over which we are taking the intersections and unions instead of writing down the precise indexing set. The variable is assumed to take any value for which the formula makes sense.) 

Let 
$$P_0=V\Big(\bigcap_{j_1}(\x0{Y_{j_1}})\cup \bigcup\limits_{j_1}Y_{j_1}\Big),$$ 
and for $k \in \mathbb N$ and $j_1, \ldots j_{k+1} \in [n]$ such that $Y_{j_l} \subseteq \Gamma_0^c(x_0, X_{j_{l-1}})$,
let
$$P_{j_{k} \ldots j_1}=V\Big(\bigcap_{j_{k+1}}(\cc0{Y_{j_{k}\ldots j_1}}\cap \x0{Y_{j_{k+1} \ldots j_1}})\cup \bigcup_{j_{k+1}} Y_{j_{k+1} \ldots j_1}\Big),$$
and let $P^i_{j_{k} \ldots j_1}$: $i=1, \ldots, c_{j_{k} \ldots j_1}$ denote the vertex sets of all the connected components of the graphs $\langle P_{j_{k} \ldots j_1} \rangle$.
From the equations above, by a formal inductive argument, one obtains that
$$V(S(w))=P_0\ \sqcup\ {\bigsqcup_{\substack{k \in \mathbb N\\ j_1, \ldots, j_{k}}}} 
P_{j_{k} \ldots j_1}=P_0\ \sqcup\ {\bigsqcup_{\substack{k\in \mathbb N\\ j_1, \ldots, j_{k}, i}}} 
P^i_{j_{k} \ldots j_1}.$$

The desired partition $\mathcal P$ consists of the set $P_0$ and the ones of the form $P^i_{j_{k} \ldots j_1}$. To see that it is of uniformly bounded diameter, 
note that for all $j_k \in [n]$,
$$\Gamma_0^c(x_0,X_{j_k}) \supseteq \cc0{X_{j_k}} \cap \bigcap_{Y_{j_{k+1}}\subseteq \cc0{X_{j_k}}}\x0{Y_{j_{k+1}}},$$
and hence 
\begin{eqnarray*}
\psi_{j_k\ldots j_1}(\Gamma_0^c(x_0,X_{j_k})) &\supseteq & \psi_{j_k\ldots j_1}\big(\cc0{X_{j_k}} \cap \bigcap_{Y_{j_{k+1}}\subseteq \cc0{X_{j_k}}}\x0{Y_{j_{k+1}}}\big)\\
&=& \cc0{Y_{j_k\ldots j_1}} \cap \bigcap_{j_{k+1}} \x0{Y_{j_{k+1}\ldots j_1}}.
\end{eqnarray*}
Therefore for any set $P^i_{j_{k} \ldots j_1}$, we have 
$$\operatorname{diam}(P^i_{j_{k} \ldots j_1}) \leq \operatorname{diam}(\Gamma_0^c(x_0,X_{j_k}))+\sum_{Y_{j_{k+1}}\subseteq \cc0{X_{j_k}}} (\operatorname{diam}(Y_{j_{k+1}})+1).$$
Since the above bound only depends on $j_k$, 
taking the maximum of the above function for all $j_k \in [n]$ and of $\operatorname{diam}(P_0)$, we obtain a uniform bound on the diameters.

It remains to show that $S(w)/\mathcal P$ is a tree. Naturally it is connected because $S(w)$ was. 
Let $e \in E(S(w))$ such that $\alpha(e) \in P^i_{j_{k} \ldots j_1}$.

Case 1: $\alpha(e) \in Y_{j_{k+1}\ldots j_1}$ for some $j_{k+1}$. 
Then either
\begin{enumerate}
	\item $\omega(e) \in V(Y_{j_{k+1}\ldots j_1}) \subseteq P^i_{j_{k} \ldots j_1}$, or
	\item $\omega(e) \in \cc0{Y_{j_{k+1}\ldots j_1}}$, then since for any $j_{k+2}$, $d(Y_{j_{k+1}\ldots j_1}, Y_{j_{k+2}\ldots j_1})=d(X_{j_{k+1}}, Y_{j_{k+2}}) \geq K$, we have
	$\omega (e) \in V(\cc0{Y_{j_{k+1}\ldots j_1}} \cap \bigcap_{j_{k+2}} \x0{Y_{j_{k+2}\ldots j_1}}) \subseteq P^j_{j_{k+1} \ldots j_1}$ for some $j$, or
	\item $\omega(e) \in \x0{Y_{j_{k+1}\ldots j_1}}$. We claim that then $\omega(e) \in V(\cc0{Y_{j_{k}\ldots j_1}} \cap \bigcap_{j} \x0{Y_{jj_{k}\ldots j_1}}) \subseteq P^i_{j_{k} \ldots j_1}$. Again, since $d(Y_{j_{k}\ldots j_1}, Y_{j_{k+1}\ldots j_1})=d(X_{j_{k}}, Y_{j_{k+1}}) \geq K$, we have $\omega (e) \in \cc0{Y_{j_{k}\ldots j_1}}$. Furthermore since for any $j \neq j_{k+1}$, $d(Y_{jj_{k}\ldots j_1}, Y_{j_{k+1}\ldots j_1}) =d(Y_j, Y_{j_{k+1}})\geq 2$, we have $\omega(e) \notin Y_{jj_{k}\ldots j_1}$. 
Note that $(\cc0{ Y_{jj_{k}\ldots j_1}} \cup  Y_{jj_{k}\ldots j_1})  \cap Y_{j_{k+1}\ldots j_1}=\emptyset$, and this implies $Y_{j_{k+1}\ldots j_1} \subseteq \x0{Y_{jj_{k}\ldots j_1}}$, in particular, $\alpha(e) \in \x0{Y_{jj_{k}\ldots j_1}}$. Let $p$ be a path in $S(w)$ from $x_0$ to $\alpha(e)$ avoiding $Y_{jj_{k}\ldots j_1}$, then $pe$ is is a path in $S(w)$ from $x_0$ to $\omega(e)$ avoiding $Y_{jj_{k}\ldots j_1}$, therefore 	$\omega(e) \in \x0{Y_{jj_{k}\ldots j_1}}$ as well. Hence $\omega(e) \in V(\cc0{Y_{j_{k}\ldots j_1}} \cap \bigcap_{j} \x0{Y_{jj_{k}\ldots j_1}}) \subseteq P^i_{j_{k} \ldots j_1}$ indeed.
\end{enumerate}

Case 2:
$\alpha(e) \in \bigcap_{j_{k+1}}(\cc0{Y_{j_{k}\ldots j_1}}\cap \x0{Y_{j_{k+1} \ldots j_1}}$.
Then either
\begin{enumerate}
	\item $\omega(e) \in V(\bigcap_{j_{k+1}}(\cc0{Y_{j_{k}\ldots j_1}}\cap \x0{Y_{j_{k+1} \ldots j_1}}) \subseteq P^i_{j_{k} \ldots j_1}$, or
	\item there exists some $j$ such that $\omega(e) \in V(Y_{jj_{k}\ldots j_1}) \subseteq P^i_{j_{k} \ldots j_1}$, or
	\item $\omega(e) \in V(Y_{j_{k}\ldots j_1})$. If $k=1$, then $\omega(e) \in P_0$. Otherwise note that as $Y_{j_{k}\ldots j_1}$ is connected, it is in a single component of $\langle P_{j_{k-1} \ldots j_1} \rangle$. Denote the vertex set of this component by $P^{i_k}_{j_{k-1} \ldots j_1}$. Then $\omega(e) \in P^{i_k}_{j_{k-1} \ldots j_1}$.
\end{enumerate}

Let $p_0:= P_0/\mathcal P$, $p^i_{j_{k} \ldots j_1}:=P^i_{j_{k} \ldots j_1}/\mathcal P$. 
Then the above argument shows us that the vertex $p^i_{j_{k} \ldots j_1}$ is adjacent to vertices of the form $p^j_{j_{k+1} \ldots j_1}$, and to either $p_0$ (in the $k=1$ case) or to $p^{i_k}_{j_{k-1} \ldots j_1}$ (in the $k\geq 2$) case. That is, $S(w)/\mathcal P$ is a tree with root $p_0$, the $(k+1)$-th level consisting of vertices of the form $p^i_{j_{k} \ldots j_1}$, and  if $k\geq 2$, then the parent of such a vertex is $p^{i_k}_{j_{k-1} \ldots j_1}$.
\qed

\bt
\label{thm:solvwp}
For the set of finitely presented, tree-like inverse monoids, the word problem is uniformly solvable.
\et

\proof
We provide an algorithm,  which takes a finite inverse monoid presentation and a word $w$ in the generators as input, halts if and only if the presentation defines a tree-like inverse monoid, and outputs $\S(w)$ when $\S(w)$ is finite, and a sapling for $\S(w)$ when it is infinite. By Theorem \ref{Ste}, this is sufficient to solve the word problem, as knowing $\S(w)$ would certainly allow us to decide $L(\A(w))$, and from a sapling for $\S(w)$, one can construct a pushdown automaton recognizing $L(\A(w))$ following the construction described in proof of Theorem \ref{cfthm}.

Let $S$ be any connected subgraph of $\exp_m(\MT(w))$ for some $m \in \mathbb N$, and suppose the image of $\MT(w)$ is contained in $S$. If there exist induced subgraphs $X_1, \ldots X_n, Y_1, \ldots, Y_n$ of $S$ which satisfy conditions (1)--(4) of Definition \ref{saplingdef}, we call $S$ a \emph{sapling candidate} with the subgraphs $X_1, \ldots ,X_n$. Note that for any connected subgraph $S$ of $\exp_m(\MT(w))$, it is decidable whether $S$ is a sapling candidate. A sapling candidate is a sapling if and only if there exists $k\in \mathbb N$ such that $\cc0{X_i} \subseteq  \exp_{k}(X_i^{+K})$. 

The algorithm is as follows. We declare a list $L$, which will contain elements of the form $(S, \exp_{l}(X_1^{+K}), \ldots, \exp_{l}(X_n^{+K}))$, where $S$ is a sapling candidate with subgraphs $X_1, \ldots ,X_n$. The list $L$ is initialized empty. We start approximating $\S(w)$ from $\MT(w)$ by iterating full $P$-expansions and determinizations. For each graph $\exp_m(\MT(w))$ obtained after $m \in \mathbb N$ iterations, perform the following steps.
\begin{enumerate}
	\item Check whether $\exp_m(\MT(w))$ is $P$-complete. If yes, output $\S(w)=\exp_m(\MT(w))$. Otherwise go to (2).
	\item Search the subgraphs of $\exp_m(\MT(w))$ for new sapling candidates. If $S$ is a newly found sapling candidate with subgraphs $X_1, \ldots ,X_n$, append $(S, \exp_0(X_1^{+K}), \ldots, \exp_0(X_n^{+K}))$ to $L$.
	\item For each $(S, \exp_{l}(X_1^{+K}), \ldots, \exp_{l}(X_n^{+K})) \in L$, check if $$\cc0{X_i} \subseteq  \exp_{l}(X_i^{+K})$$ holds for each $i=1,\ldots, n$. 
	If yes, output the sapling $S$. Otherwise replace \linebreak $(S, \exp_{l}(X_1^{+K}), \ldots, \exp_{l}(X_n^{+K}))$ by $(S, \exp_{l+1}(X_1^{+K}), \ldots, \exp_{l+1}(X_n^{+K}))$.
\end{enumerate}

Clearly, if the algorithm halts, it does indeed output either $\S(w)$ or a sapling for $\S(w)$, and in either case, by Theorem \ref{esaptreelike}, $\S(w)$ is tree-like. What remains to be shown is that if $\S(w)$ is tree-like, then the algorithm halts.

If $\S(w)$ is finite, then $\S(w) =\exp_m(MT(w))$ for some $m \in N$, and we get a `yes' after $m$ iterations in case $1$. If $\S(w)$ is infinite tree-like, then there exists a sapling for $S$, and as $S$ is a finite connected subgraph of $\S(w)$, there exists some $m \in M$ such that $S$ is a subgraph of $\exp_m(MT(w))$. Suppose $S$ is a sapling with respect to subgraphs $X_1, \ldots X_n, Y_1, \ldots, Y_n$, in particular suppose $\cc0{X_i} \subseteq  \exp_{k}(X_i^{+K})$ for $i=1,\ldots, n$. Then after at most $m$ iterations, we find that $S$ is a sapling candidate, and append $(S, \exp_0(X_1^{+K}), \ldots, \exp_0(X_n^{+K}))$ to $L$, and after at most $k+1$ more iterations, we have $(S, \exp_{k}(X_1^{+K}), \ldots, \exp_{k}(X_n^{+K}))$ in the list for which we get a `yes' at step 3. This completes the proof.
\qed

\section{
Hyperbolic Sch\"{u}tzenberger graphs  
}
\label{Sec.hyperbolic}

Recall that a graph $\Gamma$ is \emph{$\delta$-hyperbolic} if all the geodesic triangles of $\Gamma$ are $\delta$-thin. 
In this section, we consider finitely presented inverse monoids with hyperbolic Sch\"utzenberger graphs. The main result of this section is that these may have undecidable word problem, as stated in the following result.
 
\begin{T}\label{thm:Rips:correct}
There is a finitely presented inverse monoid $M$
and a positive real number $\delta > 0$ such that   
\begin{enumerate}
\item[(i)] every Sch\"{u}tzenberger graph of $M$ is  $\delta$-hyperbolic, but 
\item[(ii)] the monoid $M$ has undecidable word problem.   
\end{enumerate}
Moreover, such examples exist where the above properties hold and in addition \emph{(iii)} all of the relations are of the form $w=1$ with $w \in A^{*} $, and \emph{(iv)} the monoid $M$ is $E$-unitary.
\end{T}

We now summarize an important result from \cite{Rips82}. The Rips construction concerns groups satisfying certain small cancellation conditions, see \cite[Chapter 5, page 240]{LS}. We shall not need the definition of the small cancellation condition $C'(\lambda)$ here but just the fact that  any finitely presented group  satisfying a metric small cancellation condition $C'(\lambda)$ for some   $0 < \lambda  \leq 1/6$, is a hyperbolic group.

\begin{T}[\cite{Rips82}]
Let $\lambda > 0$ and let $G$ be any finitely presented group. Then there is a short exact sequence of groups 
\[
1 \rightarrow K \rightarrow H \xrightarrow{\phi} G\rightarrow 1
\]
such that 
\begin{enumerate} 
\item $H$ is a finitely presented group which has a presentation satisfying the small cancellation condition $C'(\lambda)$, and
\item $K$ is finitely generated (as a group). 
\end{enumerate}
\end{T}

\begin{Corol}[Corollary (b), \cite{Rips82}]\label{cor:Rips82}
There is a finitely presented hyperbolic group $H$ with a fixed finitely generated subgroup $K$ such that the subgroup membership problem for $K$ within $H$ is undecidable. 
\end{Corol}

\subsection{Tree of $\delta$-hyperbolic graphs}
In this section we will prove a key result needed to prove Theorem~\ref{thm:Rips:correct} that may be informally described as saying that ``a tree of $\delta$-hyperbolic graphs is $\delta$-hyperbolic''; see Theorem~\ref{thm:tree:of:hyperbolic}. This theorem is possibly already known, but since we could not find a proof in the literature, we provide a proof here for the sake of completeness.  


As already observed above, Sch\"{u}tzenberger graphs of inverse monoids have edges in inverse pairs.

%

\begin{T}\label{thm:tree:of:hyperbolic}
  Let $\Gamma$ be a symmetric graph. Suppose that there is a partition $\P$ of $V(\Gamma)$ with parts $P_i$ $(i \in I)$, so that $V(\Gamma) = \bigcup_{i \in I} P_i$ , such that for all $i, j \in I$ with $i \neq j$ there is at most one edge $e \in E(\Gamma)$ with $\alpha(e) \in P_i$ and $\omega(e) \in P_j$. Let $\Gamma_i = \langle P_i \rangle$ be the subgraph of $\Gamma$ induced by $P_i$ for all $i \in I$.  If $\Gamma_i$ is $\delta$-hyperbolic for all $i \in I$, and the quotient graph $\Gamma / \P$ is a tree, then the graph $\Gamma$ is $\delta$-hyperbolic. 
\end{T}
\proof
For all $i, j \in I$ if there is an edge from $P_i$ to $P_j$ then by assumption it is unique, and we shall use $e_{i,j}$ to denote this edge. So $e_{i,j}$ is the unique edge with $\alpha(e_{i,j}) \in P_i$ and $\omega(e_{i,j}) \in P_j$, when such an edge exists. Since the graph is symmetric, if the edge $e_{i,j}$ exists then so does the edge $e_{i,j}^{-1}$, and thus $e_{i,j}^{-1}$ will be the unique edge from $P_j$ to $P_i$, i.e. $e_{i,j}^{-1} = e_{j,i}$. We shall call these edges $e_{i,j}$ the \emph{transition edges} of the graph $\Gamma$.  

Observe that for all $i \in I$ the natural inclusion maps $\Gamma_i \rightarrow \Gamma$ define isometric embeddings of these graphs when viewed as metric spaces with the usual distance metric. This is because the conditions that the quotient graph is a tree, and there is a most one edge between any two parts of the partition, imply that any simple path (and hence in particular any geodesic) in $\Gamma$ between two vertices in some part $P_j$ must be completely contained in $\Gamma_i$.

Now consider a geodesic triangle $p_1, p_2, p_3$ in $\Gamma$. Let $T$ denote this triangle. 
Let $P_k$ $(k \in L)$ be the set of all parts of the partition with which $T$ has non-empty intersection. Clearly $L$ is finite. 
We shall prove that $T$ is $\delta$-hyperbolic by induction on $|L|$. 

When $|L|=1$, say $L = \{k\}$, this means that $T$ is contained in $\Gamma_k$. Since the natural embedding of $\Gamma_k$ into $\Gamma$ is an isometric embedding, it follows that $T$ is a geodesic triangle in $\Gamma_k$, and hence $T$ is $\delta$-thin in $\Gamma_k$ as $\Gamma_k$ is $\delta$-hyperbolic by assumption. Since $\Gamma_k$ embeds isometrically, and $T$ is $\delta$-thin in $\Gamma_k$, it then follows that $T$ is $\delta$-thin in $\Gamma$, as required. 

For the inductive step, now suppose that $|L|>1$ and that the result holds for all geodesic triangles whose vertex sets intersect fewer than $|L|$ parts of the partition $P$. Consider the image $Q$ of the triangle $T$ in the quotient graph $\Gamma / P$. So $Q$ is the graph with vertex set $\{ P_k : k \in L \}$ and with an edge from $P_{k_1}$ to $P_{k_2}$ if and only if there is a transition edge $e_{k_1,k_2}$ in $\Gamma$ from $\alpha(e_{k_1,k_2}) \in P_{k_1}$ to $\omega(e_{k_1,k_2}) \in P_{k_2}$. 

It follows from the hypotheses that $Q$ is a finite subtree of the tree $\Gamma/P$. In fact, $Q$ is a geodesic triangle in the tree $\Gamma/P$ and hence $Q$ is in fact a tripod.  Since we are assuming that $T$ is not contained in a single part $P_q$, that is $|L|>1$, it now follows that there exists some $P_l$  with $l \in L$ such that 

\begin{enumerate}
\item[(a)] $P_l$ is a leaf in the finite tree $Q$ (meaning a vertex of degree one), and in addition
\item[(b)] $P_l$ contains exactly one corner of the triangle $T$. 
\end{enumerate}

By the corners of the triangle $T$ we mean the set $\mathcal{C} = \{ \alpha(p_1), \alpha(p_2), \alpha(p_3) \}$. Let $P_m$ be the unique vertex in $Q$ that is adjacent to $P_l$. 
Suppose that for instance $\alpha(p_1) \in P_l$. Note that by the choice of $P_l$ we then have $\alpha(p_2) \not\in P_l$ and $\alpha(p_3) \not\in P_l$, since $\alpha(p_1) \in P_l$ and $P_l$ contains exactly one corner of the triangle $T$.

We can now decompose the sides $p_1,  p_2,   p_3$ of the triangle $T$ as follows (see Figure \ref{fig:paths} for illustration): 

\begin{align*}
p_1 & = q_1 e_{l,m} q_1',  \quad p_3 = q_3 e_{m,l} q_3', 
\end{align*}
where 
\begin{itemize}
\item $q_1$ is a geodesic in $\Gamma_l = \langle P_l \rangle$ from $\alpha(p_1)$ to $\alpha(e_{l,m})$, 
\item $q_3'$ is a geodesic in $\Gamma_l = \langle P_l \rangle$ from $\alpha(e_{l,m})$ to $\alpha(p_1)$,
\item $q_1'$ is a geodesic in $\langle \left( \bigcup_{k \in L} P_k \right) \setminus P_l \rangle$ from $\omega(e_{l,m})$ to $\alpha(p_2)$, 
\item $q_3$ is a geodesic in $\langle \left( \bigcup_{k \in L} P_k \right) \setminus P_l \rangle$ from to $\alpha(p_3)$ to $\omega(e_{l,m})$, and
\item the geodesic $p_2$ is contained in $\langle \left( \bigcup_{k \in L} P_k \right) \setminus P_l \rangle$. 
\end{itemize}

\begin{center}
\begin{figure}[h]
\includegraphics[scale=0.7]{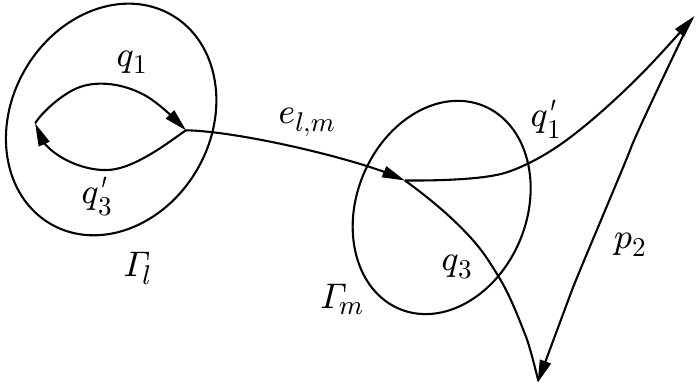}
\caption{The decomposition of $p_1, p_2, p_3$}
\label{fig:paths}
\end{figure}
\end{center}

But now $q_1', p_2, q_3$ is a geodesic triangle $T'$ in $\Gamma_l$ with no vertex in $P_l$. Hence by induction $T'$ is $\delta$-hyperbolic. Also $q_1, q_3', \iota_{\alpha(p_1)}$ is a $\delta$-hyperbolic triangle $T''$ in $\Gamma_i$, where $\iota_{\alpha(p_1)}$ denotes the empty path at $\alpha(p_1)$. 

It follows from these observations that $T$ is a $\delta$-hyperbolic triangle. To see this, let $v$ be an arbitrary vertex of $T$. There are several cases to consider. 

If $v$ is in $p_2$ then since $T'$ is $\delta$-hyperbolic, it follows that $v$ is $\delta$-close to a vertex on $q_1'$ or $q_3$, and hence is $\delta$-close to a vertex on $p_1$ or $p_3$, as required. Similarly if $v$ is in either $q_1'$ or in $q_3$, then using the fact that $T'$ is $\delta$-hyperbolic we deduce that $v$ is $\delta$-close to a vertex on $p_2$ or $p_3$, in the first case, and $\delta$-close to a vertex on $p_1$ or $p_2$, in the second case. 

The remaining possibilities are that $v$ is in $q_1$ or in $q_3'$. If $v$ is in $q_1$ then since $T''$ is $\delta$-hyperbolic $v$ is $\delta$-close to a vertex on $q_3'$ and hence to a vertex in $p_3$. Similarly, if $v$ is on $q_3'$ then $\delta$-hyperbolicity of $T''$ implies that $v$ is $\delta$-close to a vertex in $q_1$ and hence to a vertex in $p_1$.   

This completes the proof that $T$ is $\delta$-hyperbolic. 
\qed

Note that the condition given in the Theorem~\ref{thm:tree:of:hyperbolic} is stronger than simply saying that $P$ is a strong tree decomposition, since in addition to $\Gamma / P$ being a tree, we also insist that between any part of the partition there can be at most one edge.

Given a finite list of words $u_1, \ldots, u_m \in (A \cup A^{-1})^*$ we define  
\[
e(u_1, u_2, \ldots, u_m) = u_1 u_1^{-1} u_2 u_2^{-1} \ldots u_m u_m^{-1}.
\]
Since this word freely reduces to the identity in the free group $\FG(A)$ it follows that this word represents an idempotent in the free inverse monoid $\FIM(A)$. In \cite{Gray19} the following result is proved. 

\begin{T}
\label{thm_ConstructionGeneral}
Let $A = \{a_1, \ldots, a_n\}$ and let $r_1, \ldots, r_m, w_1, \ldots w_k \in (A \cup A^{-1})^*$. Let $G$ be the group $\Gpres{A}{r_1=1, \ldots, r_m=1}$  and let $M$ be the inverse monoid 
\[
\Ipres{A,t}{er_1=1, r_2=1, \ldots, r_m=1} 
\]
where $e$ is the idempotent word 
\[
e(a_1, \ldots, a_n, tw_1t^{-1}, \ldots, tw_kt^{-1}, a_1^{-1}, \ldots, a_n^{-1}). 
\]
Let $T$ be the submonoid of $G$ generated by $W = \{w_1, \ldots, w_k\}$. Then $M$ is an $E$-unitary inverse monoid. Furthermore,  if $M$ has decidable word problem then the membership problem for $T$ within $G$ is decidable. 
\end{T}

\begin{Def}
Let $G$ be a finitely generated group with finite generating set $A$ and let $t \not\in A$. Let $\mathcal{G}$ be the Cayley graph of the group $G \ast FG(t)$ with respect to $A \cup \{t\}$. We say that an induced subgraph $\Omega$ of $\mathcal{G}$ is \emph{$G$-closed} if the vertex set of $\Omega$ is closed under the action of $G$ on $G \ast FG(t)$ via right multiplication. That is, the vertex set of $\Omega$ is a union $V(\Omega) = \cup_{k \in Z} kG$ of cosets of $G$ in $G \ast FG(t)$ for some (possibly infinite) subset $Z$ of $G \ast FG(t)$. 
\end{Def}

\begin{T}\label{thm:two:triangle} 
Let $G$ be a finitely generated group with finite generating set $A$ and let $t \not\in A$.  Let $\Omega$ be an induced subgraph of the Cayley graph $\mathcal{G}$ of the group $G \ast FG(t)$ with respect to $A \cup \{t\}$.  If the Cayley graph of $G$ with respect to $A$ is $\delta$-hyperbolic and $\Omega$ is $G$-closed then $\Omega$ is $\delta$-hyperbolic.          
  \end{T}
\begin{proof}
Observe that the graph $\Omega$ satisfies the conditions of Theorem~\ref{thm:tree:of:hyperbolic} where the partition $P$ of $V(\Omega)$ is the equivalence relation with equivalence classes the $G$-orbits $vG$ ($v \in V(\Omega)$), and the transition edges are those labelled by $t$ and $t^{-1} $. 
Now for every $v \in V(\Omega)$ the induced subgraph $\langle vG \rangle$ is $\delta$-hyperbolic since it is isometric to $\langle G \rangle$ which in turn is $\delta$-hyperbolic by the assumption that the Cayley graph of $G$ with respect to $A$ is $\delta$-hyperbolic. The fact that $\langle vG \rangle$ and $\langle G \rangle$ are isometric can be proved by first noting that without loss of generality (in the case that $vG \neq G$) we can choose $v$ to be an element whose normal form in the free product $G \ast FG(t)$ ends in $t^j$ for some $j \neq 0$, and then checking that the left translation map $g \mapsto vg$ defines an isometry from $\langle G \rangle$ to $\langle vG \rangle$.   
The result now follows from Theorem~\ref{thm:tree:of:hyperbolic}.
\end{proof}

\begin{T}\label{thm:ThmA}
Let $M = \Ipres{A,t}{ f u_1=1, \ldots, u_m=1 }$ such that $u_i \in (A \cup A^{-1})^*$ for $1 \leq i \leq m$,  
$f \in (A \cup A^{-1} \cup \{t, t^{-1} \})^*$ is an idempotent word,  
and where every $a \in A$ represents an invertible element of $M$.  Set $H = \Gpres{A}{u_1=1, \ldots, u_m=1}$.  If $M$ is $E$-unitary and $H$ is $\delta$-hyperbolic then every Sch\"{u}tzenberger  graph of $M$ is $\delta$-hyperbolic.  
\end{T}
\begin{proof}
Set $B = A \cup \{t\}$. 
  Let $u \in (B \cup B^{-1})^*$ be arbitrary and consider its Sch\"{u}tzenberger graph $\mathcal{S}(u)$.   Let $R_u$ be the $\mathcal{R}$-class of $u$ in $M$. Since every $a \in A \cup A^{-1}$ represents an invertible element of $M$ it follows that for all $m \in R_u$ and for all $a \in A \cup A^{-1}$ we have $ma \in R_u$. This implies that for all $y \in R_u \sigma$ and all $a \in A$ we have $ya \in R_u \sigma$. 
Let $G$ be the maximal group image of $M$, noting that $G \cong H \ast FG(t)$ since $f$ is an idempotent word.    
By Lemma~\ref{lem:full:sibgraph}, $\mathcal{S}(u)$ is isomorphic to the subgraph $\Omega$ of the Cayley graph $\mathcal{G}$ induced by the set $R_u \sigma$. Hence $\Omega$ is an $H$-closed induced subgraph of $\mathcal{G}$. It then follows by Theorem~\ref{thm:two:triangle} that $\Omega$ is $\delta$-hyperbolic and hence so is $\mathcal{S}(u)$. This completes the proof.               
\end{proof}

We are now in a position to prove the main theorem of this section. 

\begin{proof}\emph{(Proof of Theorem~\ref{thm:Rips:correct})}
Let $A = \{ a_1, \ldots, a_n \}$ and let $r_1, \ldots, r_m, w_1, \ldots, w_k \in (A \cup A^{-1})^*$ be chosen so that $H = \Gpres{A}{r_1=1, \ldots, r_m=1}$ is a $\delta$-hyperbolic group and the membership problem for the submonoid (which is equal to the subgroup) generated by $\{ w_1, \ldots, w_k, w_1^{-1}, \ldots, w_k^{-1}\}$ is undecidable. Such a choice is possible by the Rips construction \cite{Rips82} Corollary~\ref{cor:Rips82}. Let $M = \Ipres{A,t}{fr_1 = 1, r_2=1, \ldots, r_m=1}$ where $f$ is the idempotent word \[e(a_1, \ldots, a_n, tw_1 t^{-1}, \ldots, tw_k t^{-1}, tw_1^{-1} t^{-1}, \ldots, tw_k^{-1} t^{-1}, a_1^{-1}, \ldots, a_n^{-1}).\] Then by \cite[Theorem~3.8]{Gray19} $M$ is an $E$-unitary inverse monoid with undecidable word problem. Note that the maximal group image of $M$ is $G = H \ast FG(t)$ since $f$ is an idempotent in the free inverse monoid on $A \cup \{ t \}$.
  By Theorem~\ref{thm:ThmA} since $H$ is $\delta$-hyperbolic and all $a \in A$ are invertible in $M$, and $M$ is $E$-unitary, it follows that every Sch\"{u}tzenberger graph of $M$ is $\delta$-hyperbolic, as required.                        

To establish (iii) in the statement of the theorem, first observe that by \cite[Lemma~3.3]{Gray19}
the presentation 
$\Ipres{A,t}{fr_1 = 1, r_2=1, \ldots, r_m=1}$ of $M$ is equivalent to the presentation  
\begin{align*}
\mathrm{Inv}\langle A, t \:|\: 
& r_1 = 1, r_2=1, \ldots, r_m=1, \\
& a_i a_i^{-1}=1, \; a_i^{-1} a_i=1 \; (1 \leq i \leq n), \\
& tw_jt^{-1} tw_j^{-1}t^{-1} = 1, \; tw_j^{-1}t^{-1} tw_jt^{-1} = 1 \; (1 \leq j \leq k)  
\rangle,
\end{align*}
in the sense that the identity map on $(A \cup A^{-1})^*$ induces an isomorphism between the inverse monoids defined by these two presentations. 
From the relations in this new presentation for $M$ we can see that $t$ is right invertible in $M$ and every generator $a_i$ is invertible. Hence we can add generators $t'$ and $a_i'$ and relations $tt'=1$, $a_i a_i'=1$ and  $a_i' a_i = 1$ for all $i = 1, \ldots, n$ and then replace every occurrence of $t^{-1}$ by $t'$, and of $a_i^{-1}$ by $a_i'$. 
Note that $tt'=1$ implies $t' = t^{-1}$.  
This will result in a presentation for $M$ with the property that every relation is of the form $\beta=1$ with $\beta \in (A \cup \{t \})^*$. This completes the proof of the theorem. 
\end{proof}

\section{Open questions}

We would like to close with some open questions related to our results.

According to a theorem of Bonk and Kleiner \cite[Theorem 1]{BK}, every hyperbolic group $G$ must satisfy exactly one of the following conditions:
\begin{itemize}
\item
$G$ is virtually free;
\item
the hyperbolic plane ${\mathbb{H}}^2$ admits a quasi-isometric embedding into the Cayley graph of $G$ (with respect to some/any finite generating set).
\end{itemize}
It is natural to ask whether a similar dichotomy exists for inverse semigroups with $\delta$-hyperbolic Sch\"utzenberger graphs. 

\bq
Does there exist a finitely presented inverse monoid $M$ such that
\begin{itemize}
\item all Sch\"utzenberger graphs of $M$ are $\delta$-hyperbolic,
\item $M$ is not tree-like,
\item no Sch\"utzenberger graph of $M$ contains a quasi-isometric image of $\mathbb H^2$?
\end{itemize}
\eq
The authors were able to construct a non-finitely presented counterexample.

As we mentioned in the introduction, by the Muller-Schupp theorem \cite{MS}, finitely generated virtually free groups are exactly the groups that have a context-free word problem. Theorem \ref{cfthm} can be regarded as a generalization of the direction that finitely generated virtually free (that is, tree-like) groups have a context-free word problem. It is natural to ask whether the converse implication also holds for inverse monoids.

\bq
If every Sch\"utzenberger automaton of a finitely presented inverse monoid has a context-free language, is
the monoid necessarily tree-like?
\eq

The next question is concerned with a natural possible way to construct examples of tree-like inverse monoids. We have mentioned that for any inverse monoid $M$, when $M$ is $E$-unitary, the Sch\"utzenberger graph of $M$ embeds into the Cayley graph of $M/\sigma$. If $M/\sigma$ is virtually free, then the Sch\"utzenberger graphs embed into a tree-like graph. This in general does not imply that they themselves are tree-like, the authors have again found non-finitely presented (but finitely generated) counterexamples. However we suspect that the Sch\"utzenberger graphs inherit the tree-like structure in the finitely presented case.

\bq
Let $M$ be any finitely presented, $E$-unitary inverse monoid with $M/\sigma$ virtually free. Is $M$ tree-like?
\eq

Lastly we close with raising further decidability questions about tree-like inverse monoids. By the results of Benois \cite{Benois} and Grunschlag \cite{Gru}, finitely generated virtually free groups have a decidable rational subset membership problem and hence decidable submonoid and subgroup membership problems, \cite{KSS} contains a proof of this statement relying on the easier direction of the Muller-Schupp theorem. On the other hand, it follows from \cite[Theorem 17]{LO} that inverse monoids defined by relations on group reducible words (which implies its Sch\"utzenberger graphs are trees \cite{MM}) have a decidable rational subset membership problem, and hence decidable submonoid and inverse submonoid membership problems. It is natural to ask whether (some of) these properties extend to tree-like inverse monoids.

Another related question is the decidability of the closed inverse submonoid membership problem. An inverse submonoid $H$ of an inverse monoid $M$ is called \emph{closed} if it is closed upwards in the natural partial order. These are exactly the inverse submonoids arising as stabilizers with respect to actions on $M$, and a modification of the Todd-Coxeter procedure using Stephen's algorithm can be used to build an automaton recognizing the languages of such closed inverse submonoids \cite[Theorem 3.3]{MSz}. Perhaps geometric arguments similar to the ones used in the paper can be used to decide the language of such an automaton.

\bq
Do tree-like inverse monoids have a decidable rational subset/submonoid/inverse submonoid/closed inverse submonoid membership problem?
\eq


\begin{thebibliography}{99}
	
	
	\bibitem{Ant} Y. Antol\'in, {\em On Cayley graphs of virtually free groups},
	Groups -- Complexity -- Cryptology, \textbf{3(2)} (2011), 301--327
	
	\bibitem{Silva} V. Ara\'ujo, P. V. Silva, {\em Geometric characterizations of virtually free groups},
	J. Algebra Appl. {\bf 16} No. 09 (2017), 175--180 
	
	\bibitem{AR} M. A. Ayyash, E. Rodaro, \emph{Context-freeness of the languages of Sch\"utzenberger automata of HNN-extensions of finite inverse semigroups}, Publ. Inst. Math. \textbf{99} (2016), 177--191
	
	\bibitem{Benois} M. Benois, \emph{Parties rationnelles du groupe libre}, C. R. Acad. Sci. Paris S\'er. A-B \textbf{269} (1969), A1188--A1190
	
	\bibitem{BMM} J-C. Birget, S. Margolis, J. Meakin, \emph{The word problem for inverse monoids presented by one idempotent relator}, Theoret. Comput. Sci. \textbf{123} (1994), 273--289
	
	\bibitem{BK} M. Bonk, B. Kleiner, \emph{Quasi-hyperbolic planes in hyperbolic groups}, Proc. Amer. Math. Soc. \textbf{133} (2005), 2491--2494
	
	\bibitem{BH} M. R. Bridson, A. Haelfiger, {\em Metric spaces of non-positive curvature}, Springer-Verlag (1999)
	
	\bibitem{CS} J. Cassaigne, P. V. Silva, \textit{Infinite words and confluent rewriting systems: endomorphism
		extensions}, Internat. J. Algebra Comput. \textbf{19}(4) (2009), 443--490
	
	\bibitem{CMP} A. Cherubini, J. Meakin and B. Piochi, \emph{Amalgams of finite inverse semigroups}, J. Algebra \textbf{285} (2005), 706--725
	
	\bibitem{CNR} A. Cherubini, C. Nuccio, E. Rodaro, \emph{Amalgams of finite inverse semigroups and deterministic context-free languages}, Semigroup Forum 85 (2012), 129--146
	
	\bibitem{DG} A. Duncan, R. H. Gilman, \emph{Word hyperbolic semigroups}, Math. Proc. Cambridge
	Philos. Soc. \textbf{136}(3) (2004), 513--524
	
	
	\bibitem{KapDru} C. Drutu, M. Kapovich, \emph{Geometric Group Theory,} Colloquium Publications \textbf{63}, AMS (2018)
	
	\bibitem{Gray19} R. Gray, {\em Undecidability of the word problem for one-relator inverse monoids via right-angled Artin subgroups of one-relator groups}, Invent. math. \textbf{219} (2020), 987--1008
	
	\bibitem{BM} R. Gray, M. Kambites, {\em Groups acting on semisymmetric spaces and quasi-isometries of monoids},
	Trans. AMS, \textbf{365}(2) (2013), 555--578
	
	\bibitem{Gru} Z. Grunschlag. Algorithms in Geometric Group Theory. PhD thesis, University of
	California at Berkeley (1999)
	
	\bibitem{HLM} S. Hermiller, S. Lindblad, J. Meakin, \emph{Decision problems for inverse monoids presented by a single sparse relator} Semigroup Forum \textbf{81} (2010), 128--144
	
	\bibitem{HU} J. Hopcroft, J. Ullman, {\em Introduction to Automata Theory, Languages, and Computation}, Addison-Wesley (1979)
	
	\bibitem{KSS} M. Kambites, P. V. Silva, B. Steinberg, \textit{On the rational subset problem for groups}, J. Algebra, \textbf{309} (2007), 622--639
	
	\bibitem{law} M. V. Lawson, {\em Inverse Semigroups: the Theory of
		Partial Symmetries}, World Scientific (1998)
	
	\bibitem{LO} M. Lohrey, N. Ondrusch, {\em Inverse monoids: Decidability and complexity of algebraic questions}, Inf. Comput. \textbf{205} (2007), 1212--1234
	
	\bibitem{LS} R. C. Lyndon, Paul E. Schupp, {\em Combinatorial group theory}, Springer-Verlag (2001)
	
	\bibitem{MM1} S. Margolis, J. Meakin, {\em Free inverse monoids and
		graph immersions},  Int. J. Algebra and Computation, {\bf 3}, No. 1
	(1993), 79-99
	
	\bibitem{MM} S. Margolis, J. Meakin, {\em Inverse monoids, trees, and context-free languages}, Trans. Amer. Math. Soc. \textbf{335}(1) (1993), 259--276
	
	\bibitem{MSz} J. Meakin, N. Szak\'acs, {\em Inverse monoids and
		immersions of $2$-complexes}, Int. J. Algebra and Computation {\bf
		25}  (2015), 301--323
	
	\bibitem{MS} D. E. Muller, P. E. Schupp, \emph{Groups, the theory of ends, and context-free languages}, J. Comput. Syst. Sci. \textbf{26} (1983), 295--310
	
	\bibitem{munn} W. D. Munn, {\em Free inverse semigroups}, Proc.
	London Math. Soc. {\bf 30} (1974), 384--404
	
	\bibitem{Rips82} E. Rips, {\em Subgroups of small cancellation groups}, Bull. London Math. Soc. {\bf 14} (1982), 45--47
	
	\bibitem{RC} E. Rodaro, A. Cherubini, \emph{Decidability of the word problem in Yamamura?s HNN-extensions of finite inverse semigroups}, Semigroup Forum \textbf{77} (2008), 163--186
	
	\bibitem{Serre} J-P. Serre, {\em Trees}, Springer-Verlag (1980)
	
	\bibitem{Ste1} J. B. Stephen, {\em Presentations of inverse monoids},
	J. Pure and Appl. Algebra {\bf 63} (1990), 81--112
	
	\bibitem{Ste2} J. B. Stephen, {\em Inverse monoids and rational subsets of related groups},
	Semigroup Forum (1993) \textbf{46}, 98--108
	
\end{thebibliography}
\end{document}